\documentclass{amsart}

\usepackage{amsfonts,amsmath,amsthm,amssymb}
\usepackage[margin=1in]{geometry}
\usepackage{wasysym}
\usepackage{graphicx}
\usepackage[x11names]{xcolor}
\usepackage[shortlabels]{enumitem}
\usepackage[normalem]{ulem}
\usepackage{tikz}
\usepackage{caption}
\usepackage{dsfont}
\usepackage{appendix}
\usepackage{hyperref}
\usetikzlibrary{calc}

\def\centerarc[#1](#2)(#3:#4:#5)
    { \draw[#1] ($(#2)+({#5*cos(#3)},{#5*sin(#3)})$) arc (#3:#4:#5); }

\newcommand{\R}{\mathbb{R}}
\newcommand{\Z}{\mathbb{Z}}
\newcommand{\N}{\mathbb{N}}

\newcommand{\bP}{\mathbb{P}}

\DeclareMathOperator{\supp}{supp}

\newtheorem{theorem}{Theorem}
\newtheorem{lemma}{Lemma}

\newtheorem{corollary}{Corollary}

\theoremstyle{definition}
\newtheorem{example}{Example}
\newtheorem{definition}{Definition}
\newtheorem{remark}{Remark}

\theoremstyle{remark}

\begin{document}

\title{Faber-Krahn inequality and heat kernel estimates on glued graphs}
\author{Emily Dautenhahn}
\thanks{Partially supported by NSF grant DMS-2054593.}
\address{Department of Mathematics and Statistics, Murray State University}
\author{Laurent Saloff-Coste}
\thanks{Partially support by NSF grants DMS-2054593 and DMS-234386.}
\address{Department of Mathematics, Cornell University}
\subjclass[2020]{Primary 60J10, 60G50, 35K08}
\keywords{Faber-Krahn inequality, heat kernel, isoperimetric inequality, glued graphs}

\begin{abstract}
Faber-Krahn functions provide lower bounds on the first Dirichlet eigenvalue of the Laplacian and are useful because they imply heat kernel upper bounds. In this paper, we are interested in Faber-Krahn functions and heat kernel estimates for a certain class of graphs consisting of ``sufficiently nice pages'' (satisfying a Harnack inequality) glued together via a ``sufficiently nice spine.'' For such graphs, we obtain a relative Faber-Krahn function in terms of the Faber-Krahn functions on the pages. The corresponding heat kernel upper bound involves the volumes on the various pages. In the case our graphs satisfy a property we call ``book-like'' and the spine is appropriately transient, we provide a matching lower bound for the heat kernel between two points on the gluing spine. 
\end{abstract}

\maketitle

\section{Introduction}

This paper aims to begin the study of Faber-Krahn functions and heat kernel estimates on a class of graphs that can broadly be described as constructed from sufficiently nice pieces glued in a sufficiently nice way. The reader can think of a copy of $\Z^5$, the five-dimensional square lattice, glued to a copy of $\Z^6$ by identifying their $x_1$-axes. Of particular importance and complication is that we wish to allow for gluing pieces along \emph{infinite} sets of vertices.

This work provides a necessary piece to obtaining heat kernel estimates on certain glued graphs to be presented in a forthcoming paper by the authors. Some of the work in this paper and such estimates may be found in the first author's PhD thesis \cite{Emily_thesis}. This discrete setting work is inspired by the work of Alexander Grigor'yan and the second author in the continuous setting of manifolds with ends described in \cite{lsc_ag_ends}, with supporting details developed in \cite{ag_lsc_extcptset, ag_lsc_hittingprob, ag_lsc_harnackstability,  ag_lsc_FKsurgery}. For another approach to this kind of problem where continuous spaces are glued in a way that does not produce a smooth manifold, see work of Chen and Lou \cite{chen_lou_BMvary}, where they considered $\R^2$ with a ``flag'' (a copy of $\R^1$ attached to the origin). This approach has been expanded upon in work by Lou \cite{lou_explicithk} and Lou and Li \cite{lou_li} which considers a ``flag'' ($\R^1$) attached to $\R^3$, as well as work by Ooi on $\R^{d_1}$ and $\R^{d_2}$ connected at a point for $d_1, d_2$ general \cite{ooi_bmvd}.  In the discrete setting, Woess studies a finite number of Markov chains glued together by a single root in Section 9 of \cite[pg. 96]{woess_rwgraphsgroups}. Here, our particular interest in allowing for gluing over infinite sets of vertices comes out of a hope to generalize work of Grigor'yan and Ishiwata presented in \cite{ag_ishiwata} concerning gluing two copies of $\R^n$ via a paraboloid of revolution. While the discrete setting is interesting in its own right, we also hope it will eventually provide a useful blueprint for generalizations in the continuous setting. This particular paper provides what can be seen as a discrete extension of the continuous Faber-Krahn surgery described by Grigor'yan and the second author in \cite{ag_lsc_FKsurgery}. 

This paper also serves to point out some of the possible complications with obtaining heat kernel estimates in this setting and areas where more work needs to be done. It is nontrivial to describe what we mean by ``gluing'' graphs together in ``sufficiently nice'' ways. The hypotheses under which certain results regarding Faber-Krahn functions hold are also, at present, more lax than those we require to obtain matching two-sided heat kernel estimates on glued graphs. However, as we will see, we cannot expect the Faber-Krahn estimate to be sharp in certain cases, which indicates another technique will be necessary to address certain interesting examples. We do not claim any of our hypotheses here are optimal. It is also interesting to note some of the slight differences in hypotheses in this paper versus in the continuous case of \cite{lsc_ag_ends}; for instance, \cite{lsc_ag_ends} does not discuss any type of uniformity because the gluing locus in that paper is compact.  

There are several competing definitions of Faber-Krahn functions (or inequalities). In general, a Faber-Krahn function is a function that provides a volume based lower bound on the first Dirichlet eigenvalue of the Laplacian on relatively compact domains. This lower bound is called a Faber-Krahn inequality. Faber-Krahn inequalities are related to, and sometimes called, isoperimetric inequalities because of the classical Faber-Krahn/isoperimetric theorem that for any bounded region $\Omega \subset \R^d,$
\[ \lambda_1(\Omega) \geq c_d |\Omega|^{-2/d},\]
where $\lambda_1(\Omega)$ is the first Dirichlet eigenvalue of the Laplacian and $d$ is the dimension. (See for instance Section 10.2 of \cite{GSurv} and the references therein.) The main differences in definitions of Faber-Krahn functions relate to precisely which Laplacian is being considered and to where the lower bound should hold (globally versus locally). In this paper, we are interested in relative Faber-Krahn functions, which provide \emph{local} lower bounds on the first Dirichlet eigenvalue. Faber-Krahn functions imply heat kernel upper bounds; see for instance \cite{ag_upperbds} (where Faber-Krahn inequalities are called $\Lambda$-isoperimetric inequalities). Like Poincar\'{e}, Nash, and Sobolev inequalities, Faber-Krahn inequalities (and other hypotheses) can be used to understand heat kernels, and such functional inequalities are often easier to work with than working directly with heat kernels. The theory of Faber-Krahn functions and their relation to heat kernels and other functional inequalities is well-developed in the work of Grigor'yan (see e.g. \cite{ag_localHarnack, ag_upperbds, ag_integralmax} and the references therein). 

The main result of this paper is contained in Theorems \ref{FK_booklike} and \ref{spine_lower_bd}, which together provide matching two-sided estimates for the the heat kernel between two points on the gluing set, under the assumptions that the glued graph is ``book-like'' and the spine satisfies an appropriate transience hypothesis. Theorem \ref{FK_booklike} gives the upper bound, which relies on the glued Faber-Krahn functions we develop in this paper; the lower bound given in Theorem \ref{spine_lower_bd} uses straightforward techniques and the results of our previous paper \cite{ed_lsc_trans}.

The rest of the paper proceeds as follows. Section \ref{notation} lays out the definitions and notations relating to graphs to be used in the rest of the paper. Section \ref{cutting_gluing} describes the cutting and gluing operations we are interested in. Section \ref{glue_FK_Section} discusses gluing Faber-Krahn functions. The key result on gluing Faber-Krahn functions is Lemma \ref{FK_glue}, which relies on the technical Lemma \ref{FK_quasiiso_lem}, which shows the stability of discrete Faber-Krahn inequalities under quasi-isometry. (The proof of Lemma \ref{FK_quasiiso_lem} is given in Appendix \ref{FK_quasi_iso}.) Section \ref{FK_to_HK} goes from Faber-Krahn functions to heat kernel upper bounds and contains an important result, Theorem \ref{FK_implies_HK}. The proof of Theorem \ref{FK_implies_HK} is rather technical and is therefore presented in Appendix \ref{FK_implies_HK_app}. Section \ref{spine_hk} uses Theorem \ref{FK_implies_HK} to obtain heat kernel estimates between two points in the gluing set (Theorems \ref{FK_booklike} and \ref{spine_lower_bd}). Section \ref{examples} applies our heat kernel bounds to several examples, including ones where we can show the Faber-Krahn upper bound is non-optimal.

\section{Notation and Definitions}\label{notation}

\subsection{General graph notation and random walks}\label{graphs_notation_section}

We begin by clarifying our definition of a graph and setting some basic notations. A graph $\Gamma = (V,E)$ is a set of vertices $V$ and edges $E$, where $E$ is a subset of the set of all pairs of elements of $V$. With this definition, $\Gamma$ is a simple graph that contains neither loops nor multiple edges. A graph is connected if a path of edges can always be found between any two vertices. \emph{Any graphs appearing will be assumed to be simple and connected unless stated otherwise.}

On a graph $\Gamma,$ the shortest path of edges between any two vertices defines the graph distance $d = d_\Gamma$. (If $\Gamma$ is not connected, then we say $d(x,y) = \infty$ if there is no path between vertex $x$ and vertex $y$.) We consider closed balls with respect to the distance $d_\Gamma$:
\[ B(x,r) = \{ y \in V : d(x,y) \leq r \} \quad \forall x \in V, \ r >0.\]

In order to study the Laplacian and the heat kernel on graphs, we consider graphs $\Gamma$ with a random walk structure given by edge weights (conductances) $\mu_{xy}$ and vertex weights (measure) $\pi(x)$ with the following properties:
\begin{itemize}
\item $\mu_{xy} \not = 0 \iff \{x,y\} \in E$ and $\mu_{xy} = \mu_{yx}$
(the edge weights are \emph{adapted to the edges} and \emph{symmetric}) 
\item $\sum_{y \sim x} \mu_{xy} \leq \pi(x) \quad \forall x \in V$ 
(the edge weights are \emph{subordinate} to the measure/vertex weights).
\end{itemize}
The notation $y \sim x$ means that the unordered pair $\{ x, y \}$ belongs to the edge set $E.$ The notation $y \simeq x$ means either $y \sim x$ or $y=x.$ We will use $V(x,r)$ to denote the volume (with respect to $\pi$) of $B(x,r).$

Under the above assumptions, define a Markov kernel $\mathcal{K}$ on $\Gamma$ via: 
\begin{equation}\label{MarkovKernel} 
\mathcal{K}(x,y) =\begin{cases} \frac{\mu_{xy}}{\pi(x)}, & x \not = y \\ 1 - \sum_{z \sim x} \frac{\mu_{xz}}{\pi(x)} , & x=y. \end{cases}
\end{equation}
With this definition loops are not allowed in $\Gamma,$ but the random walk is allowed to stay in place, with the probability of staying in place computable in (\ref{MarkovKernel}). Notice the Markov kernel $\mathcal{K}$ is \emph{reversible} with respect to the measure $\pi$, that is,
\[ \mathcal{K}(x,y) \pi(x) = \mathcal{K}(y,x) \pi(y) \quad \forall x, y \in \Gamma.\]
The random walk structure on $\Gamma$ may be equivalently defined by either a given $\mu, \pi,$ in which case $\mathcal{K}$ is as in (\ref{MarkovKernel}), or by a given Markov kernel $\mathcal{K}$ with reversible measure $\pi.$ 

Let $\mathcal{K}^n(x,y)$ denote the $n$-th convolution power of $\mathcal{K}(x,y)$ and $(X_n)_{n \geq 0}$ denote a random walk on $\Gamma$ driven by $\mathcal{K}$. Then $\mathcal{K}^n$ satisfies $\mathcal{K}^n(x,y) = \bP^x(X_n = y).$ The quantity $\mathcal{K}^n(x,y)$ is not symmetric (in particular, $\mathcal{K}$ itself need not be symmetric), so we study instead the \emph{transition density} of $\mathcal{K}$, which is the \emph{heat kernel} of the random walk. The heat kernel is 
\[ p(n,x,y) = p_n(x,y) = \frac{\mathcal{K}^n(x,y)}{\pi(y)}. \]

There are various hypotheses one may make about the weights of graphs that have nice consequences. In this paper, we will assume graphs have controlled weights and are uniformly lazy.

\begin{definition}[Controlled weights]\label{cont_weights}
We say a graph $\Gamma$ has \emph{controlled weights} if there exists a constant $C_c > 1$ such that 
\begin{equation}\label{controlled} \frac{\mu_{xy}}{\pi(x)} \geq \frac{1}{C_c} \quad \forall x \in \Gamma, \ y \sim x .\end{equation}
\end{definition}

This assumption implies that $\Gamma$ is locally uniformly finite (that is, there is a uniform bound on the degree of any vertex) and that for $x \sim y,$ we have $\mu_{xy} \approx \pi(x) \approx \pi(y)$, where the symbol $\approx$ means comparable up to constants that do not depend on $x,y$, as an abuse of Definition \ref{approx_def} below. 

\begin{definition}[The notation $\approx$]\label{approx_def}
For two functions of a variable $x,$ the notation $f \approx g$ means there exist constants $c_1, c_2$ (independent of $x$) such that 
\[ c_1 f(x) \leq g(x) \leq c_2 f(x).\]
\end{definition}

\begin{definition}[Uniformly lazy]\label{unif_lazy}
We say a pair $(\mu, \pi)$ is \emph{uniformly lazy} if there exists $C_e \in (0,1)$ such that 
\[ \sum_{y \sim x} \mu_{xy} \leq (1-C_e) \pi(x) \quad \forall x \in V.\]
We say a Markov kernel $\mathcal{K}$ is \emph{uniformly lazy} if there exists $C_e \in (0,1)$ such that 
\[ \mathcal{K}(x,x) = 1 - \sum_{z \sim x} \frac{\mu_{xz}}{\pi(x)} \geq  C_e \quad \forall x \in \Gamma. \]
\end{definition}

These two conditions are equivalent. In this case, the Markov chain is aperiodic, that is, $\text{gcd}\{n \geq 1: \mathcal{K}^n(x,x) >0 \} = 1.$

\emph{Unless stated otherwise, we will consider all random walk structures appearing to be uniformly lazy and have controlled weights.}

\begin{definition}[(Lazy) Simple random walk]\label{lazy_SRW}
Let $\Gamma$ be a locally uniformly finite graph. Setting $\pi(x) = \text{deg}(x)$ for all $x \in \Gamma$ and $\mu_{xy} = 1$ for all $\{x,y\} \in E$ is called taking \emph{simple weights} and defines the \emph{simple random walk} (SRW) on $\Gamma.$ Then $\mathcal{K}(x,y) = \frac{1}{\text{deg}(x)}$ if $y \sim x$ and zero otherwise. That is, at each step, the simple random walk moves to a neighbor of the current vertex with equal probability. While the simple random walk has controlled weights, it is not uniformly lazy. 

The \emph{lazy simple random walk} (lazy SRW) takes weights  $\pi(x) = 2\text{deg}(x)$ for all $x \in \Gamma$ and  $\mu_{xy} = 1$ for all $\{x,y\} \in E.$ At each step, the lazy simple random walk stays in place with probability $1/2$ and otherwise moves to a neighbor of the current vertex with equal probability. The lazy simple random walk satisfies the hypotheses of having controlled and uniformly lazy weights.   
\end{definition}

From the above definition, it is clear it is always possible to impose a random walk structure of the type we consider on any locally uniformly finite graph (any graph with a uniform bound on vertex degree). 

\subsection{Subgraphs of a larger graph}\label{subgraph}

It is often useful to consider $\Gamma$ as a subgraph of a larger graph $\widehat{\Gamma} =(\widehat{V}, \widehat{E}).$ If given $\widehat{\Gamma},$ then for any subset of $V$ of $\widehat{V},$ we can construct a graph $\Gamma$ with vertex set $V$ and edge set $E$ where $\{ x , y \} \in E$ if and only if $x,y \in V$ and $\{x,y\} \in \widehat{E}.$ We will abuse notation and use the same symbol to denote both a subset of $\widehat{V}$ and its associated subgraph. 

Further, a subgraph $\Gamma$ inherits a random walk structure from $\widehat{\Gamma}.$ We set $\pi_{\Gamma}(x) = \pi_{\widehat{\Gamma}}(x)$ and $\mu^{\Gamma}_{xy} = \mu^{\widehat{\Gamma}}_{xy}$ for all $x,y \in V, \ \{x,y\} \in E.$ (Hence we may simply use $\pi, \mu$ without indicating the whole graph versus the subgraph, provided that $x,y \in \Gamma.$) 

Then we may define a Markov kernel on $\Gamma$ as in (\ref{MarkovKernel}); we call this Markov kernel the Neumann Markov kernel of $\Gamma$ (with respect to $\widehat{\Gamma}$) and denote it by $\mathcal{K}_{\Gamma, N}.$ To be precise,

\begin{equation}
\mathcal{K}_{\Gamma,N}(x,y) = \begin{cases} \frac{\mu^\Gamma_{xy}}{\pi(x)} = \frac{\mu_{xy}^{\widehat{\Gamma}}}{\pi(x)}, & x \not = y, \ x,y \in V \\ 1 - \sum_{z \sim x} \frac{\mu^{\Gamma}_{xz}}{\pi(x)} = 1 - \sum_{z \sim x, \, z \in V} \frac{\mu^{\widehat{\Gamma}}_{xz}}{\pi(x)}, & x=y \in V.\end{cases}
\end{equation}

We can also define the Dirichlet Markov kernel of $\Gamma$ (with respect to $\widehat{\Gamma}$) by
\begin{equation}
\mathcal{K}_{\Gamma,D}(x,y) = \mathcal{K}_{\widehat{\Gamma}}(x,y) \mathds{1}_{V}(x) \mathds{1}_{V}(y) = \begin{cases} \frac{\mu^\Gamma_{xy}}{\pi(x)}, & x \not = y, \ x, y \in V \\ 1 - \sum_{z \sim x, z \in \widehat{V}} \frac{\mu^{\widehat{\Gamma}}_{xz}}{\pi(x)} , & x=y \in V,\end{cases}
\end{equation}
where $\mathds{1}_V(x) = 1$ if $x \in V$ and zero otherwise. When $V \not = \widehat{V}, \ \mathcal{K}_{\Gamma, D}$ is only a sub-Markovian kernel. 

A subgraph $\Gamma$ comes with its own notion of distance $d_\Gamma,$ where $d_\Gamma(x,y)$ is the length of the shortest path between $x$ and $y$ of edges contained in $\Gamma.$ It is always true that $d_{\widehat{\Gamma}}(x,y) \leq d_\Gamma(x,y).$ 

There are two natural notions for the boundary of a subgraph $\Gamma \subseteq \widehat{\Gamma}.$

\begin{definition}[Boundary of a subgraph]
The \emph{(exterior) boundary} of $\Gamma \subseteq \widehat{\Gamma}$ is the set of points that do not belong to $\Gamma$ with neighbors in $\Gamma,$
\[ \partial \Gamma = \{ y \in \widehat{\Gamma} \setminus \Gamma: \exists x \in \Gamma \text{ s.t. } d_{\widehat{\Gamma}}(x,y) = 1\}.\]
By contrast, the \emph{inner boundary} of $\Gamma$ is the set of points inside $\Gamma$ with neighbors outside of $\Gamma,$ 
\[ \partial_I \Gamma = \{ x \in \Gamma: \exists y \not \in \Gamma \text{ s.t. } d_{\widehat{\Gamma}}(x,y) = 1\}.\]
\end{definition} 

\section{Set-up: Cutting and gluing graphs, book-like graphs}\label{cutting_gluing}

In this section, we describe what we mean by ``gluing" graphs. The purpose of this section is to make precise an appropriate notion to take the place of the phrase ``manifolds with ends'' in the current discrete setting where we wish to allow gluing over infinite sets of vertices. 

Recall that graphs $(\Gamma, \pi,\mu)$ or $(\Gamma, \mathcal{K}, \pi)$ are assumed to be connected, uniformly lazy, and to have controlled weights unless indicated otherwise. We describe the kind of graphs we consider from the perspective of both cutting and gluing. 

\subsection{Gluing graphs together}

Consider graphs $(\Gamma_i, \pi_i, \mu^i)$ for $i=1,\dots, l$. These graphs will be referred to as ``pages'' (occasionally ``pieces'') and play the role of ``ends'' in the manifold with ends setting.

Further assume we have a graph $(\Gamma_0, \pi_0, \mu^0).$ We will refer to $\Gamma_0$ as the ``spine'' or ``gluing set''. In the language of the manifolds with ends setting, the spine plays the role of the central compact set. \textbf{We allow for $\Gamma_0$ to be disconnected}, but we still require that weights on $\Gamma_0$ be adapted, subordinate, controlled, and uniformly lazy  as described in Section \ref{graphs_notation_section}. 

We glue together the pages $\Gamma_i, \ 1 \leq i \leq l$ along the spine $\Gamma_0$ as follows: Assume each $\Gamma_i$ comes with a marked set of vertices (a ``margin''); label this set of vertices $G_i.$ Further, assume that for each $i \in \{1, \dots, l\}$ the spine $\Gamma_0$ contains a set of vertices $G_i'$ that is a copy of $G_i.$ (Note the $G_i'$ need not be disjoint.) For each $i,$ we assume there is an identification map (bijection) between $G_i$ and $G_i'$. Moreover, this identification map should satisfy the compatibility condition that there exist constants $c_I, C_I$ such that if $x \in G_i, x' \in G_i'$ are identified, then 
\begin{align}\label{compatible_weights} c_I \pi_0(x') \leq \pi_i(x) \leq C_I \pi_0(x').\end{align}
In other words, the weights at vertices that are glued together should be comparable, with uniform constants that apply to the whole graph. Due to the controlled weights condition, this implies the weights on any edges glued together are also comparable. We call the graph resulting from gluing the pages to the spine according to these identification maps $\Gamma.$

\begin{remark}~
\begin{itemize}
\item The identification maps above are between vertices. We can think of these identification maps as literally identifying vertices as being the same, or we can think of them as putting an edge between the vertices that are identified. Either way makes no real difference. However, for the sake of consistency, in this paper we think of the identification as literally identifying the vertices, which ensures the following makes sense. 

\item \textbf{In the instance that a vertex $x \in \Gamma_0$ has no neighbors (in $\Gamma_0$), it may be convenient on occasion to allow for $\pi_0(x) = 0.$}  In terms of the Markov kernel, nothing changes; if $\pi_0(x) = 0,$ then $\mathcal{K}(x,y) = 0$ for all $y \not = x$ and $\mathcal{K}(x,x) = 1.$  In this setting, obviously (\ref{compatible_weights}) cannot hold, and the appropriate compatibility condition is instead that there exist (uniform) constants $c_I, C_I$ such that for all $x_i \in \Gamma_i$ and $x_j \in \Gamma_j$ where there exists $x' \in \Gamma_0$ such that both $x_i,\ x_j$ are identified with $x',$ we have 
\begin{align}\label{compatible_weights_B} c_I \pi_j(x_j) \leq \pi_i(x_i) \leq C_I \pi_j(x_j).\end{align}
\end{itemize}
\end{remark}

The random walk structure on the glued graph $\Gamma = (V,E)$ is given by a pair of weights $(\pi, \mu)$ defined from the weights on the pages and spine via
\[ \pi(x) = \sum_{i=0}^l \pi_i(x) \quad \text{ and } \quad \mu_{xy} = \sum_{i=0}^l \mu_{xy}^i \quad \forall x, y \in V,\]
where we take the convention that $\pi_i(x) = 0$ if $x \not \in V_i$ (the vertex set of $\Gamma_i$) and $\mu_{xy}^i =0$ if $\{ x,y \} \not \in E_i$ (the edge set of $\Gamma_i$). 

We ask that the following be true of the glued graph $\Gamma$:
\begin{enumerate}[A.]
\item $\Gamma$ is connected.
\item There exists a number $\alpha>0$ such that $\Gamma_0,$ seen as a subgraph of $\Gamma,$ is $\alpha$-connected, that is, $[\Gamma_0]_\alpha,$ the $\alpha$-neighborhood of $\Gamma_0$ in $\Gamma,$ is connected.

\item When seen as a subgraph of $\Gamma,$ for all $i=1,\dots, l,$ we have $\partial_I \Gamma_i = G_i.$

\item The identification between $G_i, \ G_i'$ satisfies the description given above, i.e. is a bijection between vertices with compatible weights. 
\end{enumerate}

\begin{lemma}\label{glue_graph}
The graph $(\Gamma, \pi, \mu)$ obtained by the gluing procedure above is of the type we consider. That is, $\Gamma = (V,E)$ is a simple connected graph with edge weights $\mu_{xy}$ that are both symmetric and adapted to the edges, and which are  subordinate to the vertex weight $\pi.$ Further, the weights on $\Gamma$ are both controlled and uniformly lazy. 
\end{lemma}

\begin{proof}

Requirement A. above forces $\Gamma$ to be connected. The graph remains simple since any edges that are ``doubled'' from gluing can still be thought of as a single edge. 

The symmetry of $\mu_{xy}$ follows from the symmetry of the $\mu_{xy}^i$. If $\mu_{xy} \not =0,$ then there exists some $i \in \{0, \dots, l\}$ such that $\mu_{xy}^i \not = 0,$ which implies $\{x,y\} \in E_i,$ and hence that $\{x,y\} \in E.$ Moreover, if $\{x,y\} \in E,$ then it must be that $\{x,y\} \in E_i$ for some $i  \in \{0, \dots, l\},$ and hence $\mu_{xy} \geq \mu_{xy}^i >0.$ Therefore $\mu$ is adapted to the edges of $\Gamma.$ 

Further, the edge weights are subordinate to the vertex weights since that is true of the spine and pages:
\[ \sum_{y \sim x} \mu_{xy} = \sum_{y \sim x} \sum_{i=0}^l \mu_{xy}^i \leq \sum_{i=0}^l \pi_i(x) = \pi(x).\]

To show $\Gamma$ has controlled weights, we use the compatibility condition (\ref{compatible_weights}) or (\ref{compatible_weights_B}). Consider any edge $\{x,y\} \in \Gamma.$ Then there must exist at least one $i \in \{0, \dots, l\}$ such that $\{x,y\} \in E_i,$ since $\Gamma$ is connected and the gluing process does not add new edges (except those that occur from identifying vertices). Then, if $C_c^i$ denotes the constant for controlled weights of $\Gamma_i,$ we have: 
\begin{align*}
\frac{\mu_{xy}}{\pi(x)} = \frac{\sum_{j=0}^l  \, \mu^j_{xy}}{\sum_{j=0}^l \, \pi_j(x)} \geq \frac{\mu^i_{xy}}{C_I l \pi_i(x)} \geq \frac{1}{C_c^i C_I l} \geq \frac{1}{C_I l \, \max_{i} C_c^i}.
\end{align*}

Similarly, that $\Gamma$ is uniformly lazy again follows from that assumption on the pages and the spine:
\begin{align*}
\sum_{y \sim x} \mu_{xy} &= \sum_{i=0}^l \sum_{y \sim x} \mu_{xy}^i \leq \sum_{i=0}^l (1 - C_e^i) \pi_i(x) = \sum_{i=0}^l \pi_i(x) - \sum_{i=0}^l C_e^i \pi_i(x) \\&\leq \pi(x) - \sum_{i=0}^l (\min_{i} C_e^i) \pi_i(x) = (1 - (\min_i C_e^i)) \pi(x). 
\end{align*}
\end{proof}

\subsection{Cutting graphs into pieces}

The gluing operation described in the previous subsection is relatively natural for graphs. It is also natural to approach the question from the point of view of cutting a larger graph $(\Gamma, \pi, \mu)$ apart into pages and a spine. 

More precisely, assume we have a graph $(\Gamma, \pi, \mu).$ Identify a set of vertices (and their induced subgraph) as $\Gamma_0.$ Remove the \textbf{vertices} in $\Gamma_0$ from $\Gamma.$ This splits the graph into connected components with trailing edges left from removing the vertices in $\Gamma_0.$ 

Assume $\Gamma_0$ is such that doing the above procedure produces a finite number of connected components, $\Gamma_1, \dots, \Gamma_l,$ all of which are infinite graphs. We ``cap'' the trailing edges in a given $\Gamma_i$ with vertices and call the set of these cap vertices $G_i.$ Then $G_i$ has a natural identification with a subset of vertices $G_i'$ of $\Gamma_0.$ 

Recall subgraphs have an inherited Neumann Markov kernel obtained by taking vertex and edge weights from the larger graph, then staying in place with the probability necessary to make a Markov kernel. To define random walk structures on $\Gamma_0,\ \Gamma_1, \dots, \Gamma_l,$ we start with this general approach, but allow for some additional flexibility. Instead of forcing the weights on $\Gamma_0,\ \Gamma_1, \dots, \Gamma_l$ to be \emph{exactly} the weights inherited from $\Gamma,$ we require that the weights on these subgraphs be compatible with (comparable to) the weights on $\Gamma.$ That is, we ask that there exist constants $C_B, c_B$ such that for all $x \in \Gamma_i,\ 0 \leq i \leq l,$
\begin{align}\label{compatible_cutting} c_B \pi(x) \leq \pi_i(x) \leq C_B \pi(x),\end{align}
and we ask a similar inequality to hold for edge weights. 

The above description takes $\Gamma_0$ to be a subgraph of $\Gamma$ induced by a set of vertices; however, it is also permissible to take $\Gamma_0$ to be a set of disconnected vertices and to allow $\pi_0(x) =0$ for such vertices. 

\begin{lemma}\label{cut_graph}
Under the cutting construction described above satisfying condition (\ref{compatible_cutting}), the random walk structures on the subgraphs $\Gamma_0,\ \Gamma_1, \dots, \Gamma_l$ are uniformly lazy and have controlled weights.
\end{lemma}

Lemma \ref{cut_graph} is the complement of Lemma \ref{glue_graph} and is proved by similarly straightforward arguments. 

\begin{remark}~
\begin{itemize}
\item With the notion of adding ``cap'' vertices, the question of whether or not to add edges between two such vertices arises. Setting $\mu_{xy}^i = \mu_{xy}$ if $x,y \in \Gamma_i$ and zero else adds back in edges that were removed. 

\item While it might seem most natural to simply take the random walk structure inherited from $\Gamma$ on the spine and pages, allowing the flexibility of having comparable weights as in (\ref{compatible_cutting}) ensures it is possible to cut a graph apart in such a way that it can be glued back together into \emph{exactly} the graph we began with. Any hypotheses assumed on the pages are stable under such perturbations of weights. 
\end{itemize}
\end{remark}

We require the following of the above cutting construction:
\begin{enumerate}[a.]
\item As already mentioned, removing $\Gamma_0$ should create a finite number of connected components, all of which are infinite.
\item In $\Gamma,$ the graph $\Gamma_0$ is $\alpha$-connected.
\item The set $\Gamma_0$ should have the property that any vertex $x \in \Gamma_0$ either has (1) all neighbors also in $\Gamma_0$ or (2) has neighbors in $\Gamma_i$ and $\Gamma_j$ with $i \not =j$ and $i, j \in \{0, 1, \dots, l\}.$ 
\end{enumerate}

Property c. above ensures that $\partial_I \Gamma_i = G_i$ and prevents selecting vertices for $\Gamma_0$ that are surrounded by other vertices from only one page. It is the analog of property C. for the gluing operation. 

We demonstrate the above construction with a simple example and postpone further examples until after we have described additional hypotheses.  

\begin{example}[Gluing half-planes]\label{glue_half}
Consider $(\Gamma_i, \pi_i, \mu_{xy}^i)$ for $i=1,2,$ where $\Gamma_1 = \{ (x,y) \in \Z^2: y\geq0\}$ is the discrete upper half-plane and $\Gamma_2 = \{ (x,y) \in \Z^2 : y\leq0\}$ is the discrete lower half-plane. Take the lazy simple random walk (Definition \ref{lazy_SRW}) on both graphs. This means that vertices away from the set $\{y=0\}$ in $\Gamma_i$ have four neighbors, so they have weight eight, and that vertices on the set $\{y=0\}$ have only three neighbors, and therefore weight six. All edges have weight one. Let the spine $\Gamma_0$ be a totally disconnected copy of $\Z$ (vertices are $\Z,$ no edges) with $\pi_0(x) \equiv \mu_{xy} \equiv 0$ for all $x,y \in \Gamma_0.$ 

Let $G_i$ be the set $\{(x,y) \in \Gamma_i: y=0\}$ and $G_i' = \Gamma_0$ for both $i=1,2.$ Then all $G_i, \ G_i'$ are copies of $\Z$ and $G_i, \ G_i'$ have a natural identification for $i=1,2.$ Figure \ref{gluing_ex} illustrates this scenario and how the random walk changes along the gluing spine $\Gamma_0$. On the left, the figure shows $\Gamma_1, \Gamma_2,$ and $\Gamma_0$ as above with vertex weights shown. On the right, when we glue these graphs, the vertex weights along the spine become twelve, while all other vertex weights remain at eight. The overlapped edges now have weight two, while all other edges still have weight one. These weights can be used to calculate the transition probabilities $\mathcal{K}(x,y)$ on the glued space. 
\begin{figure}[t]
\centering
\begin{tikzpicture}[scale=.8]
   \foreach \i in {0,...,5}{
        \filldraw[blue] (\i,0) circle (4pt);
        \filldraw[blue] (\i, 1) circle(4pt);
        \node[anchor=north] at (\i, -0.1) {\small $6$};
        \node at (\i-0.2, 0.7) {\small $8$};}
    \foreach \i in {0,...,5}{
       \draw[blue,thick] (\i,0)--(\i,1)--(\i,2);}
      \draw[blue,thick](-1,1)--(6,1);
    \draw[blue,thick] (-1,0)--(6,0);
    
    \node at (6, 0.5) {\textcolor{blue}{$\Gamma_1$}};
    
        \foreach \i in {0,...,5}
       \filldraw[orange] (\i,-1) circle (4pt);
       
       \node at (6, -1) {\textcolor{orange}{$\Gamma_0$}};
    
    \foreach \i in {0,...,5}{
       \filldraw[magenta] (\i,-2) circle (4pt);
       \filldraw[magenta] (\i,-3) circle (4pt);
       \node[anchor=south] at (\i, -1.9) {\small $6$};
        \node at (\i-0.2, -2.7) {\small $8$};}
    \foreach \i in {0,...,5}{
       \draw[magenta,thick] (\i,-2)--(\i,-3)--(\i,-4);}
    \draw[magenta,thick] (-1,-2)--(6,-2);
     \draw[magenta,thick] (-1,-3)--(6,-3);
    
       \node at (6, -2.5) {\textcolor{magenta}{$\Gamma_2$}};

            \foreach \i in {10,...,15}{
       \draw[blue, thick] (\i,-1)--(\i,0)--(\i,1);}
       
             \foreach \i in {10,...,15}{
       \draw[magenta, thick] (\i,-1)--(\i,-2)--(\i,-3);}
       
       \draw[blue, ultra thick] (9,-.97)--(16,.-.97);
       \draw[magenta, ultra thick] (9,-1.03)--(16,.-1.03);
       
       \draw[blue, thick] (9,0)--(16,0);
       \draw[magenta, thick] (9,-2)--(16,-2);

         \foreach \i in {10,...,15}{
        \filldraw[orange] (\i,-1) circle (4pt);
        \filldraw[blue] (\i, 0) circle (4pt);
        \filldraw[magenta] (\i, -2) circle (4pt);
        \node at (\i-0.3, -0.7) {\small $12$};
        \node at (\i-0.3, 0.3) {\small $8$};
        \node at (\i-0.3, -1.7) {\small $8$};}
        
        \node at (16, -0.5) {$\Gamma$};

\end{tikzpicture}
 \caption{This figure illustrates gluing two half-planes with lazy SRW along a copy of $\Z$ as described in Example \ref{glue_half}. The numbers shown are vertex weights. On the left, the orange spine vertices have weight zero, so it is not shown. All edges have weight one. On the right, the ``doubled'' edges from the gluing now have weight two, while all other edges retain weight one.}\label{gluing_ex}
 \end{figure}

 Alternatively, we could cut $\Gamma = \Z^2$ with lazy SRW apart as shown in Figure \ref{half_plane_cut}. First, moving from top left to top right, we cut the half plane apart along spine $\Gamma_0 = \{y=0\}$ seen as a disconnected set of vertices with $\pi_0 \equiv 0,$ so $\Gamma_0$ is not shown in the figure. This leaves trailing vertices, which we can ``cap,'' and we add back the edges between the capped vertices. Taking the Neumann random walk on these subgraphs of the original graph $\Gamma$ gives the bottom left of Figure \ref{half_plane_cut}. Note the slight difference between this and the left side of Figure \ref{gluing_ex}. Moreover, gluing the graphs in the bottom left of Figure \ref{half_plane_cut} back together produces a variation on the right side of Figure \ref{gluing_ex}, which is not \emph{exactly} the lazy simple random walk on $\Z^2$ that we started with. 
 
However, using the flexibility of our cutting and gluing construction, it is possible to cut $\Z^2$ with lazy SRW apart in such a way that gluing it back together yields $\Z^2$ with lazy SRW.  To do so, cut $\Z^2$ apart along $\Gamma_0=\{y=0\}$ with the rule that $\Gamma_0$ is still a disconnected set of vertices of weight zero, but now each vertex in $G_i$ (the set of ``cap'' vertices) is given weight four and that edges along the $y=0$ axis are given weight $1/2$ in $\Gamma_i, \ i=1,2$. All other vertices retain weight eight and all other edges retain weight one. This is illustrated in the bottom right of Figure \ref{half_plane_cut}. This cutting/gluing operation satisfies conditions a., b., c. as well as the comparable weights condition (\ref{compatible_cutting}), and, when glued back together, produces $\Z^2$ with the lazy SRW. 

\begin{figure}[t]
\centering
\begin{tikzpicture}[scale=.8]
   \foreach \i in {0,...,5}{
           \filldraw[blue] (\i,-1) circle (4pt);
        \filldraw[blue] (\i,0) circle (4pt);
        \filldraw[blue] (\i, 1) circle(4pt);
        \node at (\i-0.2, 1.3) {\small $8$};
          \node at (\i-0.2, 0.3) {\small $8$};
            \node at (\i-0.2, -0.7) {\small $8$};}
    \foreach \i in {0,...,5}{
       \draw[blue,thick] (\i,-2)--(\i,-1)--(\i,0)--(\i,1)--(\i,2);}
      \draw[blue,thick](-1,1)--(6,1);
    \draw[blue,thick] (-1,0)--(6,0);
        \draw[blue,thick] (-1,-1)--(6,-1);
    
    \node at (6, 0.5) {\textcolor{blue}{$\Gamma$}};
    
    \foreach \i in {10,...,15}{
    \filldraw[magenta] (\i, 1) circle (4pt);
    \draw[magenta,thick] (\i,0.2)--(\i,2);}
    \draw[magenta, thick] (9,1)--(16,1);
    \node at (16, 1.5) {\textcolor{magenta}{$\Gamma_1$}};

        \foreach \i in {10,...,15}{
    \filldraw[orange] (\i, -1) circle (4pt);
    \draw[orange,thick] (\i,-0.2)--(\i,-2);}
    \draw[orange, thick] (9,-1)--(16,-1);
        \node at (16, -1.5) {\textcolor{orange}{$\Gamma_2$}};

%
    \foreach \i in {0,...,5}{
       \filldraw[magenta] (\i,-4) circle (4pt);
       \filldraw[magenta] (\i,-5) circle (4pt);
          \draw[magenta,thick] (\i,-5)--(\i,-4)--(\i,-3);
          \node at (\i-0.2, -3.7) {\small $8$};
             \node at (\i-0.2, -4.7) {\small $8$};}
       \draw[magenta, thick] (-1,-5)--(6,-5);
       \draw[magenta, thick] (-1,-4)--(6,-4);
           \node at (6, -3.5) {\textcolor{magenta}{$\Gamma_1$}};

    \foreach \i in {0,...,5}{
       \filldraw[orange] (\i,-6) circle (4pt);
       \filldraw[orange] (\i,-7) circle (4pt);
          \draw[orange,thick] (\i,-6)--(\i, -8);
                    \node at (\i-0.2, -6.3) {\small $8$};
             \node at (\i-0.2, -7.3) {\small $8$};}
       \draw[orange, thick] (-1,-6)--(6,-6);
       \draw[orange, thick] (-1,-7)--(6,-7);
                  \node at (6, -7.5) {\textcolor{orange}{$\Gamma_2$}};


                      \foreach \i in {10,...,15}{
       \filldraw[magenta] (\i,-4) circle (4pt);
       \filldraw[magenta] (\i,-5) circle (4pt);
          \draw[magenta,thick] (\i,-5)--(\i,-4)--(\i,-3);
          \node at (\i-0.2, -3.7) {\small $8$};
             \node at (\i-0.2, -4.7) {\small $4$};
             \node at (\i-0.5, -5.2) {\tiny $1/2$};}
       \draw[magenta, thick] (9,-5)--(16,-5);
   \draw[magenta, thick] (9,-4)--(16,-4);
        \node at (16, -3.5) {\textcolor{magenta}{$\Gamma_1$}};

    \foreach \i in {10,...,15}{
       \filldraw[orange] (\i,-6) circle (4pt);
       \filldraw[orange] (\i,-7) circle (4pt);
          \draw[orange,thick] (\i,-6)--(\i, -8);
                    \node at (\i-0.2, -6.3) {\small $4$};
             \node at (\i-0.2, -7.3) {\small $8$};
              \node at (\i-0.5, -5.7) {\tiny $1/2$};}
      \draw[orange, thick] (9,-6)--(16,-6);
             \draw[orange, thick] (9,-7)--(16,-7);
                \node at (16, -7.5) {\textcolor{orange}{$\Gamma_2$}};

\end{tikzpicture}
 \caption{Begin with the lazy simple random walk on $\Z^2$ in the top left. Vertex weights are shown, and all edges have weight one. We cut it apart by removing the set of vertices where $y=0$, leaving us the two graphs with trailing edges on the top right. (The ``removed''  $\Gamma_0$ vertices are treated as a disconnected set all with weight zero, and are not shown.) By capping the trailing edges with vertices, adding back edges between these vertices, and taking the Neumann random walk, we obtain the bottom left pair of graphs with given vertex weights and all edge weights one. Compare with Figure \ref{gluing_ex}. The pair of graphs on the bottom right, where all unlabeled edges have weight one, can be glued back together into $\Z^2$ with the lazy simple random walk.}\label{half_plane_cut}
 \end{figure} 

 \end{example}

\subsection{Further assumptions: Fixed width spines and book-like graphs}\label{graph_assump}

Above, we gave a very general description of cutting and gluing graphs. In this section, we describe two additional hypotheses that restrict the behavior of the spine. The general Faber-Krahn results in Sections \ref{glue_FK_Section} and \ref{FK_to_HK} hold for graphs that have a ``fixed width'' spine. A stricter hypothesis on the spine is that the graph be ``book-like.'' While we hope to have results for graphs with fixed width spines, many computations for obtaining heat kernel estimates simplify for book-like graphs (see Section \ref{spine_hk}), and Faber-Krahn estimates are not sharp for general graphs with fixed width spines (see Example \ref{cross_ex}).

\begin{definition}[Fixed width spine]\label{fixed_width}
Assume we have a construction $(\Gamma,\pi,\mu)$ as in the previous subsections with pages $\Gamma_1, \dots, \Gamma_l$ and gluing spine $\Gamma_0.$ We say the spine $\Gamma_0$ has a \emph{fixed width} of $\delta>0$ if, for all $v \in \Gamma_0,$ there exists $i \in \{1, \dots, l\}$ such that $d(v, \Gamma_i) \leq \delta.$ 
\end{definition}

\begin{definition}[Book-like graph]
Assume we have a construction $(\Gamma,\pi,\mu)$ as in the previous subsections with pages $\Gamma_1, \dots, \Gamma_l$ and gluing spine $\Gamma_0.$ For $\delta >0,$ we say $\Gamma$ is a \emph{$\delta$-book-like} graph (or simply a book-like graph) if, for all $v \in \Gamma_0$ and for all $1 \leq i \leq l,$ we have $d(v, \Gamma_i) \leq \delta.$
\end{definition}

If $\Gamma$ has a spine of fixed width $\delta,$ then each vertex in the spine $\Gamma_0$ is at distance at most $\delta$ from \emph{some} page; the nearby page may be different for different points in the spine. By contrast, if $\Gamma$ is a $\delta$-book-like graph, then each vertex in the spine $\Gamma_0$ is at distance at most $\delta$ from \emph{all} pages. If $\Gamma$ is a $\delta$-book-like graph, then its spine also has fixed width $\delta$.

\subsection{Examples}

In light of the two hypotheses on the spine described above, we describe several more examples. We are not interested in graphs where the spine does not have fixed width. 

\begin{example}[Gluing lattices along a shared lower-dimensional lattice]\label{gluing_latticeex}
The motivating or prototypical example is that of gluing lattices along a lower dimensional lattice. For $i=1, \dots, l,$ let $\Gamma_i = \Z^{D_i}$ with the lazy simple random walk, so that the pages of the graph are lattices of varying dimensions. Let the spine $\Gamma_0$ be a disconnected (no edges) copy of $\Z^k$, where we require $\min_{1 \leq i \leq l} D_i -k \geq 3.$ (This condition, which is related to transience, will become clear later.)  Then we may set $G_i$ to be equal to a copy of $\Z^k$ in $\Z^{D_i},$ for example, $G_i = \{ (x_1, \dots, x_k, 0, \dots, 0) \in \Z^{D_i}: x_1, \dots, x_k \in \Z \}$ and $G_i' = \Gamma_0$ for all $i = 1, \dots, l.$ This is a $\delta$-book-like graph for any $\delta>0$ since the spine is part of all of the pages. 

\begin{remark}
The precise definition of the spine is not so crucial, and there is a significant amount of flexibility in this kind of construction. Consider now cutting apart the graph of Example \ref{gluing_latticeex} so that the spine is a \emph{slab} made up of several copies of $\Z^k$ neighboring a central one. One way to achieve this is to define the sets 
\[ N_i =  \{ x_1, \dots, x_k, x_{k+1}, \dots, x_{D_i} \in \Z^{D_i}: x_1, \dots, x_{D_i} \in \Z \text{ and } -h \leq x_{k+1}, \dots, x_{D_i} \leq h \} \]
for $i=1, \dots, l.$ Then define the pages to be $\tilde{\Gamma}_i = \Z^{D_i} \setminus N_i$ for $i = 1, \dots, l.$ The margins $G_i$ are naturally the inner boundary of $\tilde{\Gamma}_i$ seen as a subgraph of $\Gamma$. Take the spine to be $\Gamma_0 = \cup_{i=1}^l (N_i \cup \partial_i \Gamma_i)$ and $G'_i=G_i.$ The effect of this is that there the spine is made up of several slabs of $\Z^k,$ one per page, which are connected by a central copy of $\Z^k$ (the one where all coordinates except for the first $k$ are zero).  With this definition of the pages and spine, the boundaries of the pages do not overlap, which may prove advantageous. Note an example like this demonstrates it is possible to have vertices in the spine that do not belong to any of the pages. (Although, as previous examples show, this need not be the case.) 
\end{remark}

\end{example}

\begin{example}[Cutting/gluing over a finite set]\label{finite_set_ex}
From a cutting perspective, start with a connected graph $(\Gamma, \mu, \pi)$ with controlled and uniformly lazy weights. Pick a finite set of vertices $\Gamma_0$ in $\Gamma.$ We require the three properties a., b., c. of a cutting construction above; since $\Gamma_0$ is finite, it must be $\alpha$-connected in $\Gamma,$ but the other two cutting hypotheses are not immediately satisfied. Since $\Gamma_0$ is finite, it is automatic that this construction produces a book-like graph.  

From the perspective of gluing, begin with graphs $(\Gamma_i, \mu^i, \pi_i)$ for $1 \leq i \leq l.$ Let $\Gamma_0$ be a finite set of size $K.$ In each page $\Gamma_i,$ let $G_i$ be a finite set of vertices of size $K.$ Then we may set $G_i' = \Gamma_0$ for all $1 \leq i \leq l,$ and assume we have a set bijection between each $G_i$ and $\Gamma_0.$ 

We may assume $\Gamma_0$ is connected and comes with the simple lazy random walk. Then certainly $\Gamma$ is connected, and $\Gamma_0$ remains connected in $\Gamma.$ Since $\Gamma_0$ is finite, the vertex weights are automatically comparable, and, moreover, it is impossible to make a graph that is not book-like via this construction. 
\end{example}

\begin{example}[A graph with fixed width spine that is not book-like]\label{samevol_notbook}
Suppose $l=3$ and $\Gamma_i = \Z^4$ for all $1 \leq i \leq 3$ with lazy simple random walk. Identify the $x_1$-axes of $\Gamma_1$ and $\Gamma_2$, and identify the $x_2$-axes of $\Gamma_2$ and $\Gamma_3.$ This example fits the given gluing construction, where $\Gamma_0$ consists of two axes that share a single vertex (we can think of $\Gamma_0$ as connected or not). This example also has a fixed width spine, since every vertex in $\Gamma_0$ is near the page $\Gamma_2$. However, this example is not a book-like graph since points far along the shared $x_1$-axis of $\Gamma_1,\ \Gamma_2$ are far from $\Gamma_3.$ This is the simplest type of example of a graph of fixed width spine that is not book-like. There are many possible variations on this example with different numbers of pages, dimensions of the pages, and dimensions of the gluing spine that are also not book-like. 
\end{example}

\section{Gluing relative Faber-Krahn functions}\label{glue_FK_Section}

Recall all graphs $(\Gamma, \mu, \pi)$ are assumed to be connected with controlled and uniformly lazy weights. The objective of this section is to prove Lemma \ref{FK_glue}, which provides a relative Faber-Krahn function for glued graphs $\Gamma$ as described in Section \ref{cutting_gluing} in terms of the relative Faber-Krahn functions of the pages $\Gamma_i$, under the hypotheses that $\Gamma$ has a spine of fixed width and an additional geometric hypothesis: that each page $\Gamma_i$ is uniform in $\Gamma$.

\begin{definition}[Uniform]\label{unif}
A subgraph $\Gamma$ of a graph $\widehat{\Gamma}$ is \emph{uniform} in $\widehat{\Gamma}$  if there exist constants $0<c_u, \, C_U < +\infty$ such that for any $x,y \in \Gamma$ there is a path $\gamma_{xy} = (x_0 = x, x_1, \dots, x_k =y)$ between $x$ and $y$ in $\Gamma$ such that 
\begin{enumerate}[(a)]
\item $k \leq C_U d_{\widehat{\Gamma}}(x,y) $
\item For any $j \in \{0, \dots, k\},$
\[ d_{\widehat{\Gamma}}(x_j, \partial \Gamma) = d_{\widehat{\Gamma}}(x_j, \widehat{\Gamma} \setminus \Gamma) \geq c_u (1 + \min \{ j, k-j\}).\]
\end{enumerate}
\end{definition}

\begin{definition}[Inner uniform]
A subgraph $\Gamma$ of $\widehat{\Gamma}$ is \emph{inner uniform} in $\widehat{\Gamma}$  if there exist constants $0<c_u, \, C_U < +\infty$ such that for any $x,y \in \Gamma$ there is a path $\gamma_{xy} = (x_0 = x, x_1, \dots, x_k =y)$ between $x$ and $y$ in $\Gamma$ such that 
\begin{enumerate}[(a)]
\item $k \leq C_U d_{\Gamma}(x,y) $
\item For any $j \in \{0, \dots, k\},$
\[ d_{\widehat{\Gamma}}(x_j, \partial \Gamma) = d_{\widehat{\Gamma}}(x_j, \widehat{\Gamma} \setminus \Gamma) \geq c_u (1 + \min \{ j, k-j\}).\]
\end{enumerate}
\end{definition}

Uniform and inner uniform domains were first introduced by Martio and Sarvas in \cite{martio_sarvas} in the 1970s and are related to John domains. They were further investigated by V\"{a}is\"{a}l\"{a} in the 1980s in \cite{vaisala_uniform} and by many other authors since. Literature of particular relevance here is work of Gyrya and Saloff-Coste \cite{bluebook} regarding Neumann and Dirichlet type heat kernel estimates in inner uniform domains in general Dirichlet spaces and of Diaconis, Houston-Edwards, and Saloff-Coste \cite{pd_khe_lsc_finiteMC, pd_khe_lsc_gambler} who make use of (inner) uniform domains in the discrete setting. It is from these latter two papers that we have taken our definitions above. 

The only difference between a uniform domain and an inner uniform domain is that uniform domains require the length of the path in $\Gamma$ to be comparable to the distance between $x$ and $y$ in the larger graph $\widehat{\Gamma},$ while inner uniform domains require the length of the path to be comparable to distance in $\Gamma.$ It is therefore obvious that all uniform domains are inner uniform. In this paper, the main hypothesis we make is \emph{uniformity}, although inner uniformity is necessary for the lower bounds in Section \ref{HK_lowerbd}.   

In general, determining whether a given set is (inner) uniform is challenging. Let $\phi : \R^{d-1} \to \R^d$ be a Lipschitz function and $\Omega = \{ x = (x_1, \dots, x_d) \in \R^d : x_d >  \phi(x_1, \dots, x_{d-1}) \}$ denote the set of points above the graph of $\phi$ in $\R^d.$ It is known that such domains are uniform in $\R^d$ the sense of the appropriate continuous space definition, and therefore $\Omega \cap \Z^d$ is also uniform in $\Z^d$. A slit two-dimensional lattice, that is $\Z^2$ with the set of vertices $\{(x,0): x >0\}$ (and all edges connected to them) removed, is the typical example of a domain that is inner uniform but not uniform. For further examples of (inner) uniform domains, see Section 1.3 and Chapter 6 of \cite{bluebook} or Section 2.2 of \cite{mathav_unif} and the references therein. In general, slits and ``bottlenecks'' are obstacles to uniformity. Despite the general difficulty, for the relatively simple examples considered in this paper, (inner) uniformity is usually fairly obvious. Further, (inner) uniform sets are plentiful; in particular \cite{rajala_unif} shows that in metric spaces satisfying certain hypotheses (doubling and quasiconvex), any bounded set can be $\varepsilon$-approximated by uniform domains. 

A Faber-Krahn function is a function that provides a lower bound (Faber-Krahn inequality) on the first Dirichlet eigenvalue of the Laplacian; this can be considered as a kind of isoperimetric inequality. The relationship between Faber-Krahn inequalities, heat kernel upper bounds, and other functional inequalities is found in the work of Grigor'yan (see e.g. \cite{ag_localHarnack, ag_upperbds, ag_integralmax}). There are several types of Faber-Krahn function; in this paper, we are interested in relative Faber-Krahn functions, which give local eigenvalue bounds. 

\begin{definition}[Relative Faber-Krahn function]
Let $(\Gamma,\pi,\mu)$ be an infinite connected graph with controlled and uniformly lazy weights. Let $B = B(z,r), \ \nu \in (0, +\infty).$ We say that a function $\Lambda$ where $(B,\nu) \mapsto \Lambda (B,\nu)$ is \emph{a relative Faber-Krahn function}  of $\Gamma$ if the following two properties hold:
\begin{enumerate}
\item $\nu \mapsto \Lambda (B,\nu)$ is non-increasing in $\nu$
\item For all balls $B \subset \Gamma,$ and any $\Omega \subset B,$ 
\[ \lambda_1(\Omega) \geq \Lambda (B, \pi(\Omega)),\]
where $\lambda_1(\Omega)$ denotes the first Dirichlet eigenvalue of the Laplacian in $\Omega.$ 
\end{enumerate}
\end{definition}

\begin{remark}
We can write $\lambda_1(\Omega)$ using the variational definition
\begin{align*}
\lambda_1(\Omega) = \inf_{\supp f \subseteq \Omega} \frac{\lVert \, |\nabla f| \, \rVert_2^2}{\lVert f \rVert_2^2}
\end{align*}
where 
\begin{align*}
&\lvert \nabla f\rvert(x) = \Big(\frac{1}{2} \sum_{y \in \Gamma} |f(x)-f(y)|^2 \, \mathcal{K}(x,y) \Big)^{1/2}\\
& \lVert \, \lvert \nabla f \rvert \, \rVert_2^2 = \sum_{x \in \Gamma} \big[\lvert \nabla f \rvert(x)\big]^2 \, \pi(x) = \frac{1}{2} \sum_{x,y \in \Gamma} |f(x)-f(y)|^2 \, \mu_{xy}\\
& \lVert f \rVert_2^2 = \sum_{x \in \Gamma} |f(x)|^2 \, \pi(x).
\end{align*}
\end{remark}

The following definition makes precise a notion of ``perturbing graphs,'' which is useful due to the variability in our cutting/gluing operation.

\begin{definition}[Quasi-isometry]\label{quasi-iso}~ Consider two graphs as metric measure spaces, $(\Gamma_1, d_1, \pi_1)$ and $(\Gamma_2, d_2, \pi_2),$ where $d_i$ denotes the graph distance and $\pi_i$ denotes a measure on vertices. (For this definition, a stochastic process/random walk structure is not needed.) We say $\Gamma_1$ and $\Gamma_2$ are \emph{quasi-isometric} if there exists a function $\Phi : \Gamma_1 \to \Gamma_2$ such that 
\begin{enumerate}[(1)]
\item There exists $\varepsilon > 0$ such that the $\varepsilon$-neighborhood of the image of $\Phi$ is equal to $\Gamma_2.$
\item There exist constants $a, b$ such that 
\[ a^{-1} d_1(x,y) - b \leq d_2(\Phi(x), \Phi(y)) \leq a d_1(x,y) \quad \forall \ x,y \in \Gamma_1.\]
\item There exists a constant $C_q >0$ such that 
\[ \frac{1}{C_q} \pi_1(x) \leq \pi_2(\Phi(x)) \leq C_q \pi_1(x).\]
\end{enumerate}
We call such a function $\Phi$ a \emph{quasi-isometry}.
\end{definition}

See Remark \ref{quasi-iso_compweights} in Appendix \ref{FK_quasi_iso} for some consequences of quasi-isometry in this situation. The following lemma verifies that relative Faber-Krahn functions of quasi-isometric graphs are similar.

\begin{lemma}[Relative Faber-Krahn functions of quasi-isometric graphs]\label{FK_quasiiso_lem}
Assume $\Gamma = (\Gamma, d, \pi, \mu),\ \widehat{\Gamma} = (\widehat{\Gamma}, \widehat{d}, \widehat{\pi}, \widehat{\mu})$ are connected, countably infinite graphs with controlled and uniformly lazy weights. Further, assume $\Gamma, \ \widehat{\Gamma}$ are quasi-isometric, and fix
a quasi-isometry $\Phi: \Gamma \to \widehat{\Gamma}$ and a choice of its inverse $\Phi^{-1}.$ 

Then there exist constants $c_1, c_2, c_3>0$ (depending on choice of quasi-isometry and the constants controlling the weights) such that if $\Lambda(B,\nu)$ is a relative Faber-Krahn function for  $\Gamma,$ a relative Faber-Krahn function for $\widehat{\Gamma}$ is given by 
\[ \widehat{\Lambda}(\widehat{B}(z, r), \nu) = c_1 \Lambda(B(\Phi^{-1}(z), c_2 r), c_3 \nu).\]
\end{lemma}

The proof of Lemma \ref{FK_quasiiso_lem} is straightforward but fairly technical and is given in Appendix \ref{FK_quasi_iso}.

\begin{definition}[Augmented pages]\label{aug_pages}
Let $(\Gamma, \pi, \mu)$ be a graph with spine $\Gamma_0$ of fixed width $\delta$ and pages $\Gamma_1, \dots, \Gamma_l$. For each $i \in \{1,\dots, l\},$ the \emph{augmented page} associated with $\Gamma_i$ will be denoted by $\widehat{\Gamma}_i$ and is defined as 
\[ \widehat{\Gamma}_i := [\Gamma_i]_\delta \cap (\Gamma_i \cup \Gamma_0),\]
where $[\Gamma_i]_\delta := \{ y \in \Gamma: d(y, \Gamma_i) \leq \delta\}$ is the $\delta$-neighborhood of $\Gamma_i.$ 

We define a random walk structure on the augmented pages $\widehat{\Gamma}_i$ by setting $\widehat{\pi}_i,\ \widehat{\mu}^i$ to take the values of $\pi, \mu$ from $\Gamma.$ In the event that the weights on $\Gamma_i $ are not precisely the same as those of $\Gamma$ (due to the compatible weights condition), we could take instead the weights from $\Gamma_i$ (seen as graphs separate from $\Gamma$). Any such variation will always produce comparable weights, so that any two such choices of weights creates quasi-isometric graphs. Consequently, the precise choice of weight does not alter the results obtained here except up to a change of constants. 
\end{definition} 

If $\Gamma$ is $\delta$-book-like, then $\Gamma_0 \subseteq \widehat{\Gamma_i}$ for all $i \in \{1, \dots, l\}.$ The $\delta$-book-like hypothesis ensures the augmented page $\widehat{\Gamma_i}$ consists of ``page $i$'' and the entire spine $\Gamma_0$. By contrast, if $\Gamma$ only has a spine of fixed width $\delta,$ the augmented pages need not contain the \emph{entire} spine. 

\begin{lemma}\label{aug_quasiiso}
Let $(\Gamma, \pi, \mu)$ have pages $\Gamma_1, \dots, \Gamma_l$ and spine $\Gamma_0$ of fixed width $\delta.$ Further assume each page $\Gamma_i$ is uniform in $\Gamma.$

Then each page $\Gamma_i$ is quasi-isometric to both its augmented version $\widehat{\Gamma}_i$ and to its $1$-neighborhood $[\Gamma_i]_1 = \{ x \in \Gamma: d(x, \Gamma_1) \leq 1\}.$ 
\end{lemma}

\begin{proof}
Simple arguments show that the inclusion map from $\Gamma_i \to \widehat{\Gamma}_i$ (or from $\Gamma_i \to [\Gamma_i]_1$) is a quasi-isometry.
\end{proof}

\begin{lemma}[Gluing Faber-Krahn functions]\label{FK_glue}
Let $(\Gamma, \mu, \pi)$ be a graph with pages $\Gamma_1, \dots, \Gamma_l$ and spine $\Gamma_0.$ Assume the spine $\Gamma_0$ has fixed width and that each page $\Gamma_i$ is uniform in $\Gamma$. Moreover, assume each page $\Gamma_i$ has a relative Faber-Krahn function $\Lambda_i, \ 1 \leq i \leq l.$ 

Let $B = B(z,r) \subset \Gamma.$ Then there exist constants $a_1, a_2, c_1, c_2, c_3 >0$ such that $\Gamma$ has a relative Faber-Krahn function given by 
\[ \Lambda(B, \nu) := \begin{cases} a_1 \Lambda_i(B,a_2 \nu) & \text{ if } B \subset \Gamma_i \text{ for some } i = 1, \dots, l \\ \overline{\Lambda}(B,v) & \text{ else,} \end{cases}\]
where
\[ \overline{\Lambda}(B,\nu) :=  c_1 \min_{i \in J_B} \min_{\alpha \in [B]_{\delta} \cap \Gamma_i \cap [\Gamma_0]_{\delta}} \Lambda_i (B_i(\alpha, c_2 r), c_3 \nu)).\]
with $J_B := \{ 1 \leq i \leq l : B \cap \widehat{\Gamma}_i \not = \emptyset\}.$ 
\end{lemma}

\begin{remark}
We would like to thank the anonymous reviewer for pointing out that the uniformity hypothesis can be relaxed in this theorem, as well as in our later heat kernel upper bounds that rely on it. The hypothesis needed for the proof is that $d_i(x,y) \approx d_\Gamma(x,y)$ for all $i=1, \dots, l.$ In other words, we simply need distances in the pages to be comparable to distances in the entire glued graph. However, we continue to use the hypothesis of uniformity since these results are part of a larger program where this is the appropriate hypothesis. Even in this paper, the lower bound of Theorem \ref{spine_lower_bd} requires an inner uniformity assumption. Under the assumption necessary for the lower bound, the uniformity hypothesis in this theorem and the weaker hypothesis that $d_i(x,y) \approx d_\Gamma(x,y)$ for all $i=1, \dots, l$ are equivalent. 
\end{remark}

\begin{proof}
Let $B = B(z,r) \subset \Gamma$ and let $\Omega \subset B.$ Let $f$ be a function supported in $\Omega.$ For fixed $B,$ the function $\Lambda$ as defined in the lemma is non-increasing in $\nu$ since each $\Lambda_i$ is. 

If $B \subset \Gamma_i$ for some $1 \leq i \leq l,$ then $B = B_i.$ Recalling the cutting/gluing construction of $\Gamma$ ensures comparable weights, we compute:
\begin{align*}
\lVert \, |\nabla f | \, \rVert_2^2 &= \frac{1}{2} \sum_{x,y \in \Gamma} |f(x)-f(y)|^2 \mu^{\Gamma}_{xy} 
\geq a_1 \frac{1}{2} \sum_{x,y \in \Gamma_i} |f(x)-f(y)|^2 \, \mu^{i}_{xy} \\
&\geq a_1 \Lambda_i(B_i, \pi_i(\Omega)) \sum_{x \in \Gamma_i} |f(x)|^2 \, \pi_i(x) \\
&\geq a_1 \Lambda_i(B_i, \pi_i(\Omega)) \sum_{x \in \Gamma} |f(x)|^2 \, \pi_{\Gamma}(x).
\end{align*}
Above $a_1$ is some constant depending on the compatible weights that may change from line to line. Additionally, due to the compatible weights condition and the fact that $\Omega \subset \Gamma_i,$ we have $\pi_i(\Omega) \approx \pi_\Gamma(\Omega),$ proving the first part of the lemma for some appropriate constant $a_2,$ which again depends on the compatible weights hypothesis. 

Now suppose that $B \not \subseteq \Gamma_i$ for all $1 \leq i \leq l.$ Consider the augmented pages $\widehat{\Gamma}_i.$ Set $\widehat{\Omega}_i := \widehat{\Gamma}_i \cap \Omega$ and define 
\[ \widehat{I}_i := \sum_{x \in \widehat{\Omega}_i} |f(x)|^2 \; \widehat{\pi}_i(x) \ \Big( = \sum_{x \in \widehat{\Gamma}_i} |f(x)|^2 \; \widehat{\pi}_i(x)\Big).\]
(While there is some ambiguity in how $\widehat{\pi}_i$ is defined, all such definitions are comparable to each other and to $\pi_\Gamma$ up to constants.) 

Since the spine $\Gamma_0$ has fixed width, $\Gamma \subseteq \bigcup_{i=1}^l \widehat{\Gamma}_i.$ If $c_B, C_B$ are the constants controlling compatible weights in the cutting/gluing, then
\[ l \sum_{x \in \Gamma} |f(x)|^2 \pi_{\Gamma}(x) \geq \frac{1}{C_B} \sum_{i=1}^l \widehat{I}_i \geq \frac{c_B}{C_B} \sum_{x \in \Gamma} |f(x)|^2 \pi_{\Gamma}(x),\]
since each point in $\Gamma$ appears in at most all $l$ of the sets $\widehat{I}_i$ and as the weights $\widehat{\pi}_i$ are all comparable to $\pi_\Gamma.$ 

Fix $\varepsilon < (c_B/l).$ Then there exists at least one $j \in \{ 1 , \dots, l\}$ such that $\widehat{I}_j \geq \varepsilon \lVert f \rVert_{2}^2.$ In particular, for such $j,$ this means $f$ is not identically zero on $\widehat{\Gamma}_j.$ Hence $B$ must intersect $\widehat{\Gamma}_j$ and $j \in J_B.$ Further, $B \cap (\widehat{\Gamma}_j \setminus \Gamma_j) \not = \emptyset,$ since $B$ is connected and is not contained only in $\Gamma_j.$  Take $y \in B \cap (\widehat{\Gamma}_j \setminus \Gamma_j).$ 

We claim there exists a constant $C^*$ such that $\widehat{\Omega}_j \subseteq \widehat{B}_j(y, C^*r).$

First, $\widehat{\Omega}_j \subseteq \Omega \subseteq B_{\Gamma}(z,r) \subset B_{\Gamma}(y, 4r)$ since $y \in B_{\Gamma}(z,r).$ Let $w \in \widehat{\Omega}_j.$ Recall $\Gamma_j$ and $\widehat{\Gamma}_j$ are quasi-isometric: the constant $\delta$ from the spine width can be used in part (1) of the definition of quasi-isometry and there is come constant $a$ in (2) related the distances of $\Gamma_j$ and $\widehat{\Gamma}_j.$ Also, $\Gamma_j$ is uniform in $\Gamma$ with a constant $c_U$ relating their distances. Then, since $r \geq 1,$ 
\begin{align*} 
\widehat{d}_j (w, y) &\leq a d_j (\Phi^{-1}(w),\Phi^{-1}(y)) \leq a c_U d_\Gamma(\Phi^{-1}(w),\Phi^{-1}(y)) \\
&\leq a c_U \Big(d_\Gamma \big(\Phi^{-1}(w), w\big) + d_\Gamma(w,y) +d_\Gamma \big(y, \Phi^{-1}(y)\big)\Big)\\
&\leq ac_U(2 \delta+ 4 r) \leq ac_U(2\delta +4)r.
\end{align*}
Thus $w \in \widehat{B}_j(y, C^* r)$ for $C^* = a c_U (2\delta+4).$

Let  $\widehat{\Lambda}_j$ denote a Faber-Krahn function for $\widehat{\Gamma}_j.$ Assume the compatible edge weights condition is given by $c_m \, \mu_{xy}^\Gamma \leq \mu_{xy}^i \leq C_M \, \mu_{xy}^\Gamma.$ Then, by the above, our choice of the index $j,$ and Lemma \ref{FK_quasiiso_lem},
\begin{align*}
\lVert \, |\nabla f|\,  \rVert_2^2 &= \frac{1}{2} \sum_{x,y \in \Gamma} |f(x)-f(y)|^2 \mu^{\Gamma}_{xy} 
\geq \frac{1}{C_M} \sum_{x,y \in \widehat{\Gamma}_j} |f(x) - f(y)|^2 \; \widehat{\mu}^j_{xy}  \\
&\geq \frac{1}{C_M} \widehat{\Lambda}_j \big(\widehat{B}_j(y, C^* r), \widehat{\pi}_j(\widehat{\Omega}_j)\big) \sum_{x \in \widehat{\Gamma}_j} |f(x)|^2 \; \widehat{\pi}_j(x)  \\
&\geq \frac{\varepsilon}{C_M} \, \widehat{\Lambda}_j \big(\widehat{B}_j(y, C^* r), \widehat{\pi}_j(\widehat{\Omega}_j)\big) \sum_{x \in \Gamma} |f(x)|^2\,  \pi_{\Gamma}(x)  \\
&\geq c_1 \Lambda_j \big(B_j (\Phi^{-1}(y), c_2 r), c_3 \widehat{\pi}_j(\widehat{\Omega}_j)\big) \lVert f \rVert_2^2.
\end{align*} 
In the last line, we combined some constants. Now, $\widehat{\pi}_j(\widehat{\Omega}_j) \leq C_B \pi(\Omega),$ and Faber-Krahn functions are non-increasing in $\nu,$ so we may replace $c_3 \widehat{\pi}_j (\widehat{\Omega}_j)$ by $c_3 C_B \pi(\Omega)$ in $\Lambda_j$ above and then call $c_3 C_B$ again $c_3$.

It remains to show that we can control the expression involving $\Lambda_j$ above by $\overline{\Lambda}.$ The vertex $\Phi^{-1}(y)$ belongs to $\Gamma_j$ and lies within the $\delta$-neighborhoods of both $\Gamma_0$ and $B$ since $y \in \Gamma_0 \cap B.$ Therefore it is clear
\[ c_1 \Lambda_j (B_j (\Phi^{-1}(y), c_2 r), c_3 \pi(\Omega)) \geq \overline{\Lambda}(B,\nu).\]
\end{proof} 

\section{From glued Faber-Krahn functions to heat kernel estimates}\label{FK_to_HK}

Lemma \ref{FK_glue} in the previous section gives a relative Faber-Krahn function for glued graphs under relatively mild assumptions. In this section, we use Lemma \ref{FK_glue} to obtain heat kernel upper bounds for $\Gamma$ in the case that the pages $\Gamma_1, \dots, \Gamma_l$ are sufficiently nice. Here the meaning of ``sufficiently nice'' is that the pages are Harnack graphs. Section \ref{Harnack_graphs_section} defines relevant additional hypotheses, and Section \ref{Harnack_FK_section} provides heat kernel estimates that follow from the glued Faber-Krahn function of Lemma \ref{FK_glue}. 

\subsection{Harnack graphs and Faber-Krahn}\label{Harnack_graphs_section}

\begin{definition}[Doubling]
A graph is said to be \emph{doubling} if there exists a constant $D$ such that for all $r>0, \ x \in \Gamma,$ 
\begin{equation}
V(x,2r) \leq D V(x,r).
\end{equation}
\end{definition}

\begin{definition}[Poincar\'{e} inequality]
We say that $\Gamma$ satisfies the \emph{Poincar\'{e} inequality} if there exist constants $C_p >0, \ \kappa \geq 1$ such that for all $r > 0,\ x \in \Gamma,$ and functions $f$ supported in $B(x,\kappa r),$ 
\begin{equation}
\sum_{y \in B(x,r)} |f(y) - f_B|^2 \, \pi(y) \leq C_p \ r^2 \sum_{y,z \in B(x, \kappa r)} |f(y)-f(z)|^2 \, \mu_{yz},
\end{equation}
where $f_B$ is the (weighted) average of $f$ over the ball $B=B(x,r),$ that is,
\[ f_B = \frac{1}{V(x,r)} \sum_{y \in B(x,r)} f(y)\, \pi(y).\]
\end{definition}

Under doubling, the Poincar\'{e} inequality with constant $\kappa \geq 1$ (which appears in $B(x, \kappa r)$ on the right hand side) is equivalent to the Poincar\'{e} inequality with $\kappa = 1.$ For a proof of this fact in the continuous case, see Theorem 5.3.4 of \cite{lsc_sobolevbook}. For the discrete case, see Corollary A.51 of \cite{Barlow_graphs}; the corollary assumes the graph has controlled weights as we do in this paper.

\begin{definition}[Harmonic function]
A function $h: \Gamma \to \R$ is \emph{harmonic} (with respect to $\mathcal{K}$) if 
\begin{equation}\label{harmonic} h(x)  = \sum_{y \in \Gamma} \mathcal{K}(x,y) h(y) \quad \forall x \in \Gamma.\end{equation}
Given a subset $\Omega$ of $\Gamma$ (usually a ball), $h$ is harmonic on that set if (\ref{harmonic}) holds for all points in $\Omega$; this requires that $h$ be defined on $\{v\in \Gamma: \exists \omega\in \Omega,\ v\simeq \omega\} = \Omega \cup \partial \Omega$. 
\end{definition}
As $\mathcal{K}(x,y) = 0$ unless $y \simeq x,$ the sum over $y \in \Gamma$ in (\ref{harmonic}) can be replaced by a sum over $y \simeq x.$ 

\begin{definition}[Elliptic Harnack inequality]
A graph $(\Gamma, \pi, \mu)$ satisfies the \emph{elliptic Harnack inequality} if there exist $\eta \in (0,1),\ C_H>0$ such that for all $r>0, \ x_0 \in \Gamma,$ and all non-negative harmonic functions $h$ on $B(x_0, r),$ we have 
\[ \sup_{B(x_0,\, \eta r)} h \leq C_H \inf_{B(x_0,\, \eta r)} h. \]

We say a graph satisfies the elliptic Harnack inequality \emph{up to scale} $R>0$ if the above definition holds for $0<r<R,$ where the constant $C_H$ depends upon $R.$ 
\end{definition}

\begin{definition}[(sub)Solution of discrete heat equation]\label{subsoln}
A function $u: \Z_{+} \times \Gamma \to \R$ \emph{solves the discrete heat equation} if
\begin{equation}\label{heat_eqn} u(n+1, x) - u(n,x) = \sum_{y \in \Gamma} \mathcal{K}(x,y)[ u(n,y) - u(n,x)]  \quad \forall n \geq 1, \ x \in \Gamma.\end{equation}
Given a discrete space-time cylinder $Q = I \times B,$ $u$ solves the heat equation on $Q$ if (\ref{heat_eqn}) holds there (this requires that for each $n\in I$, $u(n,\cdot)$ is defined on $\{z\in \Gamma: \exists x \in B,\ z\simeq x\} = B \cup \partial B$).

Similarly, a function $u$ is a \emph{subsolution} of the discrete heat equation on space-time cylinder $Q$ if for each $n \in I, \ u(n, \cdot)$ is defined on $\{z\in \Gamma: \exists x \in B,\ z\simeq x\} = B \cup \partial B$  and $u$ satisfies (\ref{heat_eqn}) with the equality sign replaced by a $\leq$. 

Notice also in (\ref{heat_eqn}) it is possible to cancel the $u(n,x)$ from both sides. 
\end{definition} 

\begin{definition}[Parabolic Harnack inequality]
A graph $(\Gamma, \pi, \mu)$ satisfies the (discrete time and space) \emph{parabolic Harnack inequality} if: there exist $\eta \in (0,1), \ 0 < \theta_1 < \theta_2 < \theta_3 < \theta_4$ and $C_{PH}>0$ such that for all $s,r >0, \ x_0 \in \Gamma,$ and every non-negative solution $u$ of the heat equation in the cylinder $Q = (\Z_{+} \cap [s, s+ \theta_4 r^2]) \times B(x_0,r),$ we have
\[ u(n_-, x_-) \leq C_{PH} \, u(n_+, x_+) \]
for all $(n_-, x_-) \in Q_-, \ (n_+, x_+) \in Q_+$ such that $d(x_-, x_+) \leq n_+ - n_- ,$ where 
\begin{align*}
&Q_- = (\Z_+ \cap [ s + \theta_1 r^2, s + \theta_2 r^2]) \times B(x_0, \eta r)\\ 
&Q_+ = (\Z_+ \cap [ s+ \theta_3 r^2, s+ \theta_4 r^2]) \times B(x_0, \eta r).
\end{align*}

We say a graph satisfies the parabolic Harnack inequality \emph{up to scale} $R>0$ if the above definition holds for $0<r<R,$ where the constant $C_{PH}$ now depends on $R$.
\end{definition}

The parabolic Harnack inequality obviously implies the elliptic version. Any graph with controlled weights satisfies the parabolic Harnack inequality at scale $1$, and therefore satisfies the parabolic Harnack inequality up to any finite scale $R$. The following theorem relates several of the above inequalities. 

\begin{theorem}[Theorem 1.7 in \cite{Delmotte_PHI}]\label{VD_PI}
Given $(\Gamma, \pi, \mu)$  (or $(\Gamma, \mathcal{K}, \pi)$) where $\Gamma$ has controlled weights and $\mathcal K$ is uniformly lazy, the following are equivalent: 
\begin{enumerate}
\item[(a)] $\Gamma$ is  doubling and satisfies the Poincar\'{e} inequality 
\item[(b)] $\Gamma$ satisfies the parabolic Harnack inequality
\item[(c)] $\Gamma$ satisfies two-sided Gaussian heat kernel estimates, that is there exists constants $c_1, c_2, c_3, c_4 >0$ such that for all $x,y \in \Gamma, \ n \geq d(x,y),$
\begin{equation}\label{2sided_graphs}
\frac{c_1 }{V(x, \sqrt{n})} \exp\Big(-\frac{d^2(x,y)}{c_3 n}\Big) \leq p(n,x,y) \leq \frac{c_3}{V(x, \sqrt{n})} \exp\Big(-\frac{d^2(x,y)}{c_4 n}\Big).
\end{equation}
\end{enumerate}
\end{theorem}

\begin{definition}[Harnack graph]\label{Harnack_graph}
If $(\Gamma, \pi, \mu)$ satisfies any of the conditions in Theorem \ref{VD_PI}, we call $\Gamma$ a \emph{Harnack graph}. 
\end{definition}

\begin{remark}
The uniformly lazy assumption avoids problems related to parity (such as those that appear in bipartite graphs). Without this assumption, it may be that (a) holds but $p(n,x,y) = 0$ for some $n \geq d(x,y),$ and then (b) and the lower bound in (c) do not hold. Here we avoid such difficulties by assuming the graph is uniformly lazy; another solution to this problem is to state (b) and (c) for the sum over two consecutive discrete times $n,\ n+1$, e.g., for (c), $p(n,x,y) + p(n+1,x,y).$ 
\end{remark}

Harnack graphs have relative Faber-Krahn functions of a particularly nice form. 

\begin{theorem}[Faber-Krahn Function of a Harnack Graph, \cite{TC_AG_GraphVol}]
If $\Gamma$ is a Harnack graph, then there exist constants $a, \alpha>0$ such that for all balls $B(z,r) \subseteq \Gamma$ and all $\nu >0,$
\begin{equation}\label{Harnack_FK}
\Lambda(B(z,r), \nu) = \frac{a}{r^2} \bigg(\frac{V(z,r)}{\nu}\bigg)^{\alpha}.
\end{equation}
Conversely, if $\Gamma$ has a relative Faber-Krahn function of form (\ref{Harnack_FK}) above, then $\Gamma$ is a Harnack graph. 
\end{theorem}

\subsection{Heat kernel upper bounds for glued Harnack graphs}\label{Harnack_FK_section}

Assume $\Gamma$ is a graph with pages $\Gamma_1, \dots, \Gamma_l$ satisfying all hypotheses of Lemma \ref{FK_glue} \textbf{and} the additional hypothesis that all pages are Harnack. Then, for each $i$, $\Gamma_i$ has a Faber-Krahn function $\Lambda_i$ of the form (\ref{Harnack_FK}) with constants $a_i, \alpha_i.$ Consequently, the function $\overline{\Lambda}$ can be written terms of volume functions of the pages: 
\begin{align*}
\overline{\Lambda}(B,\nu) &=  c_1 \min_{i \in J_B}\, \min_{y \in [B]_{\delta} \cap \Gamma_i \cap [\Gamma_0]_{\delta}} \Lambda_i (B_i(y, c_2 r), c_3 \nu) \\
&=  c_1 \min_{i \in J_B}\, \min_{y \in [B]_{\delta} \cap \Gamma_i \cap [\Gamma_0]_{\delta}} \frac{a_i}{c_2^2 r^2} \,\bigg( \frac{V_i(y, c_2r)}{c_3 \nu}\bigg)^{\alpha_i} \\[1ex]
&\approx \min_{i \in J_B}\, \min_{y \in [B]_{\delta} \cap \Gamma_i \cap [\Gamma_0]_{\delta}} \frac{C}{r^2} \, \bigg( \frac{V_i(y,r)}{\nu}\bigg)^{\alpha_i}.
\end{align*}
The last line uses that each $V_i$ is doubling and $\Gamma$ has a finite number of pages, so that $V_i(y, c_2 r) \approx V_i(y, r)$ (since $c_2$ is a fixed constant). 

Define 
\begin{equation}\label{V_min}
V_{\min}(z,r) := \min_{i \in J_B} \min_{y \in [B]_{\delta} \cap \Gamma_i \cap [\Gamma_0]_{\delta}} V_i(y,r)
\end{equation}
and set
\begin{align}\label{fcn_F}
F(z,r) := \begin{cases} V_i(z, r) = V_{\Gamma}(z,r), & \text{ if } B \subset \Gamma_i \text{ for some } 1 \leq i \leq l \\
V_{\min}(z,r), & \text{ else.} \end{cases}
\end{align}

\begin{corollary}\label{FKform_Harnackpg}
Let $(\Gamma, \mu, \pi)$ be a graph with pages $\Gamma_1, \dots, \Gamma_l$ and a fixed width spine $\Gamma_0$. Assume each page $\Gamma_i$ is Harnack and is uniform in $\Gamma.$ Set $\alpha = \min_{1 \leq i \leq l} \alpha_i,$ where $\alpha_i$ is the exponent appearing in the relative Faber-Krahn function of $\Gamma_i$ as in (\ref{Harnack_FK}). Let $V_{\min}, \ F$ be defined as in (\ref{V_min}), (\ref{fcn_F}), respectively. Then a relative Faber-Krahn function for $\Gamma$ is given by
\begin{equation}\label{FK_withF}
\Lambda(B(z,r), \nu) = \frac{C}{r^2} \Big(\frac{F(z,r)}{\nu}\Big)^\alpha.
\end{equation}
\end{corollary}

The corollary follows directly from Lemma \ref{FK_glue} and the above observations. 

\begin{theorem}\label{FK_implies_HK}
Assume $\Gamma$ is a graph with a relative Faber-Krahn function given by (\ref{FK_withF}) as in Corollary \ref{FKform_Harnackpg} above. Then for all $x,y \in \Gamma, \ n >0,$
\begin{align}\label{HK_upper_FK}
p(n,x,y) \leq \frac{c_1}{\sqrt{F(x, \sqrt{n}) F(y, \sqrt{n})}} \exp\Big(-\frac{d^2(x,y)}{c_2 n}\Big).
\end{align}
\end{theorem}

There is a significant amount of literature on the relationship between Faber-Krahn functions and heat kernel upper bonds, particularly in the work of Alexander Grigor'yan and various coauthors. However, the literature does not quite cover the case given here.  Proving Theorem \ref{FK_implies_HK} from Corollary \ref{FKform_Harnackpg} requires several technical lemmas and small modifications of existing work; the proof is given in Appendix \ref{FK_implies_HK_app}.

\section{The \texorpdfstring{$\Gamma$}{glued graph} heat kernel from \texorpdfstring{$\Gamma_0$}{the spine} to \texorpdfstring{$\Gamma_0$}{the spine}} \label{spine_hk}

While the heat kernel estimate of Theorem \ref{FK_implies_HK} holds for any $x,y \in \Gamma,$ in practice we expect this upper bound to be optimal only in certain regimes. In particular, we want to use Theorem \ref{FK_implies_HK} to obtain upper bounds on the spine-to-spine heat kernel, that is, on $p(n,v,w)$ where both $v$ and $w$ belong to the spine $\Gamma_0.$ However, even for this spine-to-spine heat kernel, the upper bound provided by Faber-Krahn is not necessarily optimal unless additional hypotheses about the nature of the gluing and pages are made. 

We show the Faber-Krahn upper estimate on the spine-to-spine heat kernel is optimal in the case of certain book-like graphs and when all pages are the same. Under some additional assumptions, we prove optimality by proving matching lower bounds using a local parabolic Harnack inequality. This spine-to-spine heat kernel estimate will be crucial for our future work obtaining full matching upper and lower heat kernel estimates on book-like graphs. The main reason we can show optimality in the book-like case (and a few others) is that the computation of $V_{\min}$ simplifies significantly.

\subsection{Heat Kernel Upper Bounds for Book-like Graphs}

Assume $\Gamma$ is a $\delta$-book-like graph with pages $\Gamma_1, \dots, \Gamma_l$ and spine $\Gamma_0.$ 

Then for any $z \in \Gamma_0, r>\delta,$ and $B=B(z,r),$ the set $J_B = \{1, 2, \dots, l\}$ since balls around $z \in \Gamma_0$ always see all pages. Therefore the function $V_{\min}$ defined by (\ref{V_min}) becomes
\begin{equation}\label{Vmin_booklike}
V_{\min}(z,r) := \min_{1 \leq i \leq l}\, \min_{y \in [B]_{\delta} \cap \Gamma_i \cap [\Gamma_0]_{\delta}} V_i(y,r),
\end{equation}

Combining Lemma \ref{FK_glue} and Theorem \ref{FK_implies_HK} yields the following theorem:
\begin{theorem}[Book-like spine-to-spine heat kernel upper bound]\label{FK_booklike}
Let $(\Gamma, \mu,\pi)$ be a $\delta$-book-like graph with pages $\Gamma_1, \dots, \Gamma_l$ and spine $\Gamma_0$ where the pages $\Gamma_i$ are Harnack and uniform in $\Gamma$ for each $i=1,\dots, l$. Then for any $v,w \in \Gamma_0$ and $m \gg d_\Gamma(v,w) + \delta,$
\begin{align}\label{UB_sqrt}
p(m,v,w) &\leq \frac{c_1}{\sqrt{V_{\min}(v,\sqrt{m}) V_{\min}(w, \sqrt{m})}} \exp\Big(-\frac{d_\Gamma^2(v,w)}{c_2m}\Big)\\
&\leq \frac{c_1}{\min_{1 \leq i \leq l} V_i(v_i, \sqrt{m})} \exp\Big(-\frac{d_\Gamma^2(v,w)}{c_2m}\Big),
\end{align}
where $v_i$ is the/a closest point in $\Gamma_i$ to $v$.
\end{theorem}

\begin{proof}
The first inequality in Theorem \ref{FK_booklike} is a restatement of Theorem \ref{FK_implies_HK} with simplifications coming from the book-like assumption and from only looking at points in the spine. 

The second inequality follows once we show that $V_{\min}(v, \sqrt{m}) \approx \min_{1 \leq i \leq l} V_i(v_i, \sqrt{m}).$ In particular, this proves that $V_{\min}$ is doubling, since each $V_i$ is doubling and the minimum of a finite number of doubling functions is doubling. Standard arguments show that, if $V_{\min}$ is doubling, the quantity $\sqrt{V_{\min}(v, \sqrt{m}) V_{\min}(w,\sqrt{m})}$ can be replaced by either $V_{\min}(v, \sqrt{m})$ or $V_{\min}(w, \sqrt{m})$ at the price of changing $c_1,\ c_2$.

Let $v_i$ be a closest point in $\Gamma_i$ to $v \in \Gamma_0.$ Since $\Gamma$ is $\delta$-book-like, $v_i$ belongs to the set $[B(v,r)]_\delta \cap \Gamma_i \cap [\Gamma_0]_\delta$ for all $r \geq 0,$ so obviously $\min_{y \in [B(v,r)]_\delta \cap \Gamma_i \cap [\Gamma_0]_\delta} V_i(y, r) \leq V_i(v_i, r).$ 

On the other hand, suppose $y \in [B(v,r)]_\delta \cap \Gamma_i \cap [\Gamma_0]_\delta.$ Then if $r \geq 1,$ $v_i \in B(y, 4 \delta r)$.  Since $V_i$ is doubling, we know $V_i(v_i, r) \leq V_i(y, 4\delta r) \leq \widehat{C} V_i(y, r)$ for some constant $\widehat{C}$ independent of $y, r.$ Hence $V_i(v_i, r) \leq \widehat{C} \min_{y \in [B(v,r)]_\delta \cap \Gamma_i \cap [\Gamma_0]_\delta} V_i(y, r).$
\end{proof}

\begin{remark}
A subtle point in the proof above that relies on the book-like hypothesis is as follows: Since in the book-like case, $J_B = \{1, \dots, l\},$ it is obvious that $V_{\min}(x,\cdot)$ is always a \emph{non-decreasing} function of $r$ (as expected for volume functions). The definition of a doubling function involves the volume of a ball, which is of course a non-decreasing function in the radius of the ball. However, $V_{\min}$ is not strictly a volume function, and in the case where $J_B$ does not always include all possible indices, the function $V_{\min}$ need not be non-decreasing in $r$ in any reasonable sense, as we will see later in Example \ref{cross_ex}. The standard arguments to replace quantities such as $\sqrt{V(x,r) V(y,r)}$ by $V(x,r)$ or $V(y,r)$ at the price of changing the constant in the exponential term depend upon $V$ being doubling and a non-decreasing function in $r$.
\end{remark}  

\subsection{Heat Kernel Lower Bounds for Book-like Graphs}\label{HK_lowerbd}

In the case of a book-like graph that satisfies an additional hypothesis that seeing the spine is relatively rare, we obtain a lower bound for the $\Gamma_0$ to $\Gamma_0$ heat kernel that matches the upper bound of Theorem \ref{FK_booklike}. This lower bound relies on three ideas. First, a local parabolic Harnack inequality holds, so heat kernels evaluated at nearby points are comparable. Second, the heat kernel on all of $\Gamma$ is bigger than the Dirichlet heat kernel on each page. Third, if the spine is uniformly $S$-transient as defined below, then the Dirichlet heat kernel on each page is comparable to the Neumann heat kernel on that page. Without such a transience hypotheses, estimating the Dirichlet heat kernel is more challenging.

It is a well-known phenomenon in the continuous case that heat kernel upper estimates obtained from Faber-Krahn are not sharp in the recurrent (parabolic) case. See for instance the discussion in work of Grigor'yan, Ishiwata, and Saloff-Coste on pg. 162 of \cite{parabolic_ends} on the difference in optimal middle estimates for manifolds with ends glued via a compact set in the case where all ends are recurrent (parabolic) versus the case where all ends are transient (non-parabolic); see also \cite{lsc_ag_ends, ag_lsc_FKsurgery}. In the case there are some recurrent and some transient pages, one can get a sharp bound by considering the related $h$-transform space (see Section 6 of \cite{lsc_ag_ends}). We discuss this phenomenon in the discrete case in Example \ref{recurrent_ex} below. 

If $K$ is a subgraph of a graph $(\Gamma, \pi, \mu)$, let $\tau_K$ denote the hitting time of $K$ and $\psi_K(x)$ denote the hitting probability of $K,$ that is 
\[ \tau_K := \min\{ n \geq 0: X_n \in K\}, \quad \psi_K(x) := \bP^x(\tau_K < + \infty).\]

The notion of uniform $S$-transience, defined by the authors in \cite{ed_lsc_trans}, is precisely what was necessary there to obtain good estimates on the Dirichlet heat kernel.

\begin{definition}[Uniform $S$-transience]
Let $(\widehat{\Gamma}, \mathcal{K}, \pi)$ be a connected graph with controlled weights and $K$ be a subset of $\widehat{\Gamma}$ such that $\Gamma := \widehat{\Gamma} \setminus K$ is connected. We say the subgraph $\Gamma$ is \emph{$S$-transient} (``survival"-transient) or that the graph \emph{$\widehat{\Gamma}$ is $S$-transient with respect to the set $K$} if there exists $x \in \widehat{\Gamma}$ such that $\psi_K(x) < 1.$ (If this is not the case, $\Gamma$ is said to be \emph{$S$-recurrent}.)

Further, $\Gamma$ is \emph{uniformly $S$-transient}, ($\widehat{\Gamma}$ is \emph{uniformly $S$-transient with respect to} $K$), if there exist $L, \varepsilon >0$ such that $d(x,K) \geq L$ implies that $\psi_K(x) \leq 1- \varepsilon.$ 
\end{definition}

\begin{lemma}[Corollary 3.14 of \cite{ed_lsc_trans}]\label{D_approx_N}
Assume that $\widehat{\Gamma}$ is a Harnack graph that is uniformly  $S$-transient with respect to $K$ and that $\Gamma := \widehat{\Gamma} \setminus K$ is inner uniform. Then there exist constants $0<c, \, C < +\infty$ such that 
\begin{align}\label{Dirichlet_approx_Neumman}
cp_{\Gamma,N}(Cn, x, y) \leq p_{\Gamma,D}(n,x,y) \leq p_{\Gamma,N}(n,x,y). 
\end{align}
\end{lemma}

See \cite{ed_lsc_trans}, particularly Section 2, for some typical examples of what subgraphs are $S$-transient and the differences between this definition and classical transience.

Let $m \in \Z_+$ and $v,w \in \Gamma_0.$ Let $v_i, w_i$ denote the closest points to $v$ and $w,$ respectively, in $\Gamma_i$ for $i=1,\dots, l.$ 

\begin{theorem}[Book-like spine to spine heat kernel lower bound]\label{spine_lower_bd}
Let $(\Gamma, \mu, \pi)$ be a book-like graph with pages $\Gamma_1, \dots, \Gamma_l$ and spine $\Gamma_0$ such that each pages $\Gamma_i$ is (1) Harnack, (2) uniform in $\Gamma$, and (3) inner uniform and uniformly $S$-transient in the augmented page $\widehat{\Gamma_i}.$

Then there exist constants $c_1, c_2$ such that for all $v, w \in \Gamma_0$ and all $m \gg d_\Gamma(v,w) + \delta,$
\begin{equation}\label{spine_lowerbd}
p(m,v,w) \geq \frac{c_1}{\min_{1 \leq i \leq l} V_i(v_i, \sqrt{m})} \exp\Big(-\frac{d_\Gamma^2(v, w)}{c_2m}\Big),
\end{equation}
where $v_i$ is a closest point to $v$ in $\Gamma_i$.
\end{theorem}

\begin{remark}
The condition $m \gg d_\Gamma(v,w) + \delta$ means $m \geq C (d_\Gamma(v,w)+\delta)$ for some fixed constant $C;$ we add the $\delta$ to handle the event that $v=w.$ This condition is needed to ensure the random walk has enough time to see all pages; for small values of $m$ bounds on $p(m,v,w)$ are not very interesting. 
\end{remark}

\begin{proof}
Since $\Gamma$ is $\delta$-book-like, $d(v_i, v), d(w_i, w)\leq \delta$ for all $i.$ The parabolic Harnack inequality holds at scale $1$ (in fact, at any finite scale) on $\Gamma$ since $(\Gamma, \mu, \pi)$ is a connected graph with controlled and uniformly lazy weights. Therefore, since $m$ is sufficiently large, 
 \[ p(m,v,w) \approx p(m', v_i, w) \approx p(m'', v_i, w_i),\]
 where $m \approx m' \approx m''.$ 
 
 By assumption, each page $\Gamma_i$ is uniformly $S$-transient and inner uniform in $\widehat{\Gamma}_i,$ a Harnack graph (since $\widehat{\Gamma}_i$ is quasi-isometric to $\Gamma_i,$ which is Harnack). Therefore, by Lemma \ref{D_approx_N}, $p_{\Gamma_i, D} \approx p_{\Gamma_i, N}$ for all $1\leq i \leq l.$ By definition, the Dirichlet heat kernel of any subgraph is less than the heat kernel on the entire graph. Hence 
 \[ p_\Gamma(m'', v_i, w_i) \geq p_{\Gamma_i, D}(m'', v_i, w_i) \geq c p_{\Gamma_i, N}(Cm'', v_i, w_i) \quad \forall 1 \leq i \leq l.\]
 
 Moreover, $\Gamma_i$ is Harnack, so $p_{\Gamma_i, N}$ satisfies a Gaussian lower-bound. Consequently,
 \begin{equation}\label{each_piece_lower}
 p_\Gamma(m,v,w) \geq \frac{c_1}{V_i(v_i, \sqrt{m})} \exp\Big(-\frac{d_{\Gamma_i}^2(v_i,w_i)}{c_2 m}\Big) \quad \forall 1 \leq i \leq l.
 \end{equation}

 The only difference between (\ref{each_piece_lower}) and (\ref{spine_lowerbd}) is the distance appearing in the exponential. Since each page $\Gamma_i$ is uniform in $\Gamma,$ $d_{\Gamma_i}(v_i, w_i) \approx d_\Gamma(v_i, w_i).$  Applying the triangle inequality and the $\delta$-book-like hypothesis,
 \begin{equation}\label{replace_dist} d_\Gamma(v_i, w_i) \leq d_\Gamma(v_i,v) + d_\Gamma(v,w) + d_\Gamma(w_i,w) \leq 2\delta + d_\Gamma(v,w).\end{equation}

 Recall $(a+b)^2 \approx a^2 +b^2$ and $\delta$ is constant. Inequality (\ref{replace_dist}) goes the appropriate way to replace $d_{\Gamma_i}^2(v_i, w_i)$ by $d_\Gamma^2(v,w)$ in (\ref{each_piece_lower}), at the price of changing the values of $c_1, c_2$. 

Since (\ref{each_piece_lower}) holds with $d_\Gamma(v,w)$ in the exponential for all $i \in \{1, \dots, l\}$, it holds for the minimum over all such $i.$
\end{proof}

\begin{remark}
If $\Gamma$ is not book-like, then there exist points in the spine that are increasingly far away from certain pages. Therefore, there is not an upper bound on $\max_{v \in \Gamma_0, 1 \leq i \leq l} d_\Gamma(v, \Gamma_i)$, and it is not possible to move points off the spine into \emph{all} pages using a fixed small scale parabolic Harnack inequality. See Example \ref{cross_ex} below for a discussion of what can go ``wrong'' when the pages do not all have the same volume.

However, in the event that all the pages are the same (or have similar volumes), for instance $\Gamma_i = \Z^d$ for all $i,$ then the above argument still holds, as does the matching upper bound from Theorem \ref{FK_booklike}. Similarly, there is no issue if there is one ``smallest'' page that all points in the spine are always near.
\end{remark}

\section{Examples}\label{examples}

In this section, we provide some concrete examples of the spine to spine heat kernel bounds. We discuss cases where the upper bound from Faber-Krahn is optimal as in Section \ref{spine_hk}, as well as examples where this is not the case. 

\begin{example}[Heuristics of a Product Space]
One of the simplest examples of a book-like graph where the spine is an infinite set of vertices is gluing three copies of $\Z^4$ with lazy simple random walk by identifying their $x_1$-axes. As a graph, this space can be viewed as a product of three copies of $\Z^3$ identified at the origin times a copy of $\Z.$ If we forget for the moment that we are dealing with discrete space and time, one expects the heat kernel on this product space to be described by the product of the heat kernels. Continuing on this heuristic reasoning, the results of the case of gluing manifolds via a compact set of \cite{lsc_ag_ends} apply to the three copies of $\Z^3$ identified at the origin, and the heat kernel on a line is classical. If $\Gamma$ is our glued graph and $\Z^3 \# \Z^3 \# \Z^3$ denotes three $\Z^3$'s identified at the origin $o$, this heuristic gives 
\begin{align*}
p_{\Gamma}(n,o,o) &= p_{\Z}(n,o,o) p_{\Z^3 \# \Z^3 \# \Z^3}(n,o, o) \approx C\frac{1}{n^{1/2}}\frac{1}{n^{3/2}} \approx \frac{C}{n^2},
\end{align*}
which matches our results. This analysis can be made rigorous for Brownian motion on three copies of $\R^4$ glued along a cylinder of radius one centered about the $x_1$-axis. 

However, discrete space and time make the above heuristic very hard to follow. There are two main difficulties. First, the product property is simply not true with discrete time. Even a simple random walk on $\Z^3$ in discrete time is not the product of three independent simple random walks on $\Z.$ The transfer of off-diagonal estimates (which we will obtain in a forthcoming paper) from continuous time to discrete time is very delicate, and there is no general method for this transfer. Second, part of our goal is to treat many variations that are not covered by this reasoning, which is sensitive to small perturbations such as adding an extra edge or changing a handful of weights. 
\end{example}

\begin{example}[Lattices glued via a lower-dimensional lattice]\label{lattice_lattice_spineest} Recall the situation described in Example \ref{gluing_latticeex} of gluing pages $\Gamma_i = \Z^{D_i}$ with lazy SRW along a shared lower dimensional lattice $\Z^k$. The requirement that $\min_{1 \leq i \leq l} D_i - k \geq 3$ ensures the pages are uniformly $S$-transient. (See Examples 2.6 and 2.13 from \cite{ed_lsc_trans}.) Lattices are Harnack, and it is straightforward to verify the uniformity hypotheses. In $\Z^{D_i},$ we have $V_i(x,r) \approx r^{D_i}$. Set $D_{\min} := \min_{1 \leq i \leq l} D_i$. Then Theorems \ref{FK_booklike} and \ref{spine_lower_bd} give 
\begin{equation} \label{lattice_lattice_est}
p(m,v,w) \approx \frac{C}{m^{D_{\min}/2}} \exp\Big(-\frac{d_\Gamma^2(v,w)}{cm}\Big) \quad \forall v, w \in \Gamma_0, \ m \gg d_\Gamma(v,w) + \delta,
\end{equation}
where we abuse $\approx$ to mean there are different constants $C,c$ in the upper and lower bounds. 

Notice that the dimension $k$ of the gluing lattice does not directly appear in this spine to spine heat kernel estimate. 
\end{example}

\begin{example}[Finite gluing set]
Consider the case $\Gamma_0$ is finite, as in Example \ref{finite_set_ex}. Then $\Gamma$ is automatically book-like. By the small scale parabolic Harnack inequality and the results of Section \ref{spine_hk},
\begin{equation}
p(m,v,w) \approx p(m,o,o) \approx \frac{C}{\min_{1 \leq i \leq l} V_i(o_i, \sqrt{m})} \quad \forall v,w,o \in \Gamma_0, m \gg \text{diam}(\Gamma_0).
\end{equation}
\end{example} 

\begin{example}[All pages have the same volume]
If all pages have the same volume, for instance as was described in Example \ref{samevol_notbook}, where three copies of $\Z^4$ with lazy SRW were joined along a cross (a copy of $\Z^4$ is joined to both the $x_1$ and $x_2$ axes of a third, central copy of $\Z^4$) then the proofs from Section \ref{spine_hk} go through without difficulty, regardless of whether the graph is book-like or not. In this case $V_{\min}=V_i$ for all $i$, so suddenly seeing a new page does not make a difference to the volume computations. 
\end{example}

\begin{example}[Smallest page always visible]\label{smallest_visible}
Suppose the pages are easily ordered in terms of their volumes, as is the case with the lattice examples considered above. (See \cite{G_I_LSC_survey} for some examples of manifolds with oscillating volume functions where ends are not easily ordered.) Let $\Gamma_1$ be the page with smallest volume. Then provided the fixed width spine $\Gamma_0$ satisfies the property that there exists $\alpha>0$ such that $d(v, \Gamma_1) \leq \alpha$ for all $v \in \Gamma_0,$ the bounds from Section \ref{spine_hk} go through. Since the smallest page is always visible, the function $V_{\min}$ is always based off volumes in the smallest page and has no sudden changes. 
\end{example}

\begin{example}[$\Z^4$ and $\Z^6$ glued to $\Z^5$ via a cross]\label{cross_ex}
This is the simplest example of a non-book-like graph where the Faber-Krahn upper bound is non-optimal. Consider the situation of Example \ref{samevol_notbook}, but where the pages have different dimensions/volumes: Take three lattices, a $\Z^4,$ a $\Z^5,$ and a $\Z^6$, all with lazy SRW. Glue the $\Z^4$ and $\Z^5$ together by identifying their $x_1$-axes. Glue the $\Z^5$ and the $\Z^6$ together by identifying their $x_2$-axes. The resulting space has a copy of $\Z^5$ as a central part, where a random walk can travel to the copy of $\Z^4$ via the $x_1$-axis or to the copy of $\Z^6$ via the $x_2$-axis. The gluing spine $\Gamma_0$ is a ``cross'', the union of the $x_1$-axis and the $x_2$-axis. 

Unlike Example \ref{smallest_visible} above, where all pages had the same volume, the proofs of Section \ref{spine_hk} no longer hold. For instance, suppose $v=o$ is the origin and $w$ is a point located somewhere along the $x_2$-axis. By Theorem \ref{FK_implies_HK}, we have the heat kernel upper bound 
\[ p(n, o, w) \leq \frac{c_1}{\sqrt{V_{\min}(o, \sqrt{n})V_{\min}(w, \sqrt{n})}} \exp \Big(-\frac{d^2(o,w)}{c_2 n }\Big),\]
where $V_{\min}$ is defined as in (\ref{V_min}). Since $o$ is the origin, $B(o,r)$ always intersects all three pages for all $r>0$ and therefore $V_{\min}(o, \sqrt{n}) \approx (\sqrt{n})^4 = n^2.$ However, $B(w, r)$ only intersects the pages $\Z^5$ and $\Z^6$ for $r< d(o, w)$. This gives 
\[ V_{\min}(w, \sqrt{n}) \approx \begin{cases} n^{5/2},& \text{if } n < d^2(o,w) \\ n^2,& \text{ if } n \geq d^2(o,w).\end{cases}\]

Note the sharp transition in $V_{\min}$ and how it is not a non-decreasing function of the radius. 

Therefore the upper bound from Theorem \ref{FK_implies_HK} gives
\[ p(n, o, w) \leq \begin{cases} \frac{c_1}{n^{9/4}} \exp\Big(-\frac{d^2(o,w)}{c_2n}\Big), & n < d^2(o,w) \\ \frac{c_1}{n^2} \exp\Big(-\frac{d^2(o,w)}{c_2n}\Big), & n \geq d^2(o,w),\end{cases}\]
where $o$ is the origin and $w$ is along the $x_2$-axis.

The upper bound is not continuous since when $n=d^2(o,w),$ the two exponents on $n$ do not match and the exponential term is constant. We do not expect this bound matches the optimal bound in either regime. 

For instance, if $o^*, \ w^*$ are closest points in $\Z^5 \setminus \Gamma_0$ to $o, \ w$, respectively,  the following lower bound holds:
\begin{align*}
p_\Gamma(n,o,w) &\geq c_1 p_\Gamma(c_2 n, o^*, w^*) \quad \text{(small scale PHI)}\\
&\geq c_1 p_{\Z^5 \cup \Gamma_0, D} (c_2 n, o^*, w^*) \quad \text{(definition of Dirichlet heat kernel)} \\
&\gtrsim p_{\Z^5}(n, o^*, w^*) \gtrsim \frac{C}{n^{5/2}} \exp\Big(-\frac{d_{\Z^5}^2(o^*,w^*)}{cn}\Big).
\end{align*}
The last line follows from \cite[Corollary 3.14]{ed_lsc_trans} since $\Z^5 \cup \Gamma_0$ is Harnack (and quasi-isometric to $\Z^5$) and $\Z^5$ is inner uniform and uniformly $S$-transient with respect to the cross $\Gamma_0$. Since $o^*, \ w^*$ are close to $o, \ w,$ the constant $c_2$ is such that $c_2 n \approx n$ and $d_{\Z^5}(o^*, w^*) \approx d_\Gamma (o, w)$ (provided $o,\ w$ are not such that $o^* = w^*$). 

We expect the above lower bound to be optimal when $n \ll d^2(o,w).$ However, $5/2=10/4 > 9/4$, which indicates the upper bound given by Faber-Krahn is too large. (Intuitively, for such small times, the above lower bound should be optimal: Since $w$ is along the $\Z^5$ and $\Z^6$ join and there is not enough time to see the $\Z^4$ page, we expect the heat kernel to look like that of $\Z^5$.)

On the other hand, if $n \gg d^2(o,w),$ we expect a random walk to spend a lot of time in the $\Z^4$ page. However, it is still necessary to move from $w$ to the $\Z^4$ page and to come out of the $\Z^4$ page to find $o$. In this case, we expect the heat kernel to ignore the $\Z^6$ page and behave like the heat kernel on a $\Z^4$ glued to a $\Z^5$ by identifying their $x_1$-axes. In this case, as we will show in forthcoming work, the expected heat kernel bound is 
\[ p(n,o,w) \approx C\Big[\frac{1}{n^2 [d(w, \Z^4)]^2} + \frac{1}{n^{5/2}}\Big] \exp\Big(-\frac{d^2(o,w)}{cn}\Big).\]
Proving the informal arguments given above would require a significant amount of work, but we hope it is nonetheless convincing to the reader that the Faber-Krahn upper bound is non-optimal in this situation, nor can we expect it to be, since the Faber-Krahn upper bound of Theorem \ref{FK_implies_HK} must hold for all pairs of points $x, y.$
\end{example}

\begin{example}[$\Z^5$ and $\Z^6$ glued along a two-dimensional parabola] Consider a copy of $\Z^5$ and a copy of $\Z^6$, both with lazy SRW. In each of these two lattices, consider the set of lattice points $P$ in the $x_1x_2$-plane that lie inside the ``parabola'' $x_2 = \pm x_1^\alpha$ for $\alpha \in (0,1)$ and $x_1 \in \Z_{\geq 0}$. 

There are two natural ways to glue these two lattices together via the parabola. One way is to identify the \emph{interiors} of the parabolas across the two lattices, so that the spine $\Gamma_0$ is the set of lattice points that lie inside or on the boundary of the parabola. We may also duplicate the parabola and connect the $\Z^5$ and the $\Z^6$ to this duplicate parabola so that the pages do not directly touch.

The other natural approach is to glue the lattices together \emph{only along the parabola itself}. In this second case, since the ``parabola'' in $\R^n$ given by $P^{n, \alpha}  = \{ x \in \R^{n} : x_2 = \pm x_1^\alpha\}$ (with $n=5, 6$) does not necessarily hit integer lattice points, consider the sets $\widetilde{P}^{n,\alpha} := \{ x \in \Z^n : d_{\R^n}(x, P^{n,\alpha}) \leq \sqrt{n}\}$. Since $d_{\R^n}(x, \Z^n) \leq \sqrt{n}$ for all $x\in \R^n,$ every point in the parabola $P^{n, \alpha}$ in $\R^n$ lies within distance at most $\sqrt{n}$ of an integer lattice point, and therefore corresponds to a point in $\widetilde{P}^{n, \alpha}$. Moreover, points more that distance $\sqrt{n}$ from each other in $\R^n$ that lie on the parabola are associated with different integer lattice points. In this way, we have a set of integer lattice points that approximates the parabola, and we may glue $\Z^5$ and $\Z^6$ by identifying the vertices in $\widetilde{P}^{5,\alpha}$ and $\widetilde{P}^{6,\alpha}.$  

If the interiors of the parabolas are glued together through an extra copy of the parabola, then the pages look like $\Z^5$ and $\Z^6$, which are Harnack. Uniformity and inner uniformity hold since the parabola is a two-dimensional set living inside a five- or six-dimensional space. Moreover, since the parabola is smaller than a plane, the pages are uniformly $S$-transient in their augmented versions. As this graph only has two pages, it is automatically book-like. Therefore the hypotheses of Section \ref{spine_hk} hold and, for any $v, w \in P$ and $m$ sufficiently large,
\begin{equation}\label{para_gluing} p(m,v,w) \approx \frac{C}{m^{5/2}} \exp\Big(-\frac{d^2(v,w)}{cm}\Big).\end{equation}
This is the same estimate as (\ref{lattice_lattice_est}) from gluing lattices via lower-dimensional lattices in Example \ref{lattice_lattice_spineest}. The estimate (\ref{para_gluing}) also holds if we glued the $\Z^5$ and the $\Z^6$ together via only the parabola itself (and not the interiors). This tells us that in some sense the ``dimension'' or shape of the spine does not directly impact the spine to spine heat kernel estimate, so long as the spine is such that the uniform $S$-transience and uniformity hypotheses hold. While these estimates are the same, \emph{where} they hold is dependent upon the shape of the spine, and the shape of the spine will be very important in future work computing full two-sided heat kernel estimates. 
\end{example}

\begin{example}[$\Z^4$ and $\Z^5$ glued via a parabola]
In the style of the previous example, now consider a copy of $\Z^4$ and of $\Z^5,$ both with lazy SRW, glued by identifying the interiors of parabola $P^{n,\alpha}$ ($\alpha \in(0,1), \ n=5,6$). While the pages are still Harnack and uniform, the spine $\Gamma_0$ is uniformly $S$-transient with respect to the $\Z^5$ page and $S$-transient but \emph{not uniformly so} with respect to the $\Z^4$ page \cite[Section 3.5]{ed_lsc_trans}. 

For $v,w \in \Gamma_0$ and $m$ sufficiently large, Faber-Krahn gives the upper bound
\[ p(m,v,w) \leq \frac{c_1}{m^2} \exp\Big(-\frac{d_\Gamma^2(v,w)}{c_2 m}\Big).\]
In the argument about the lower bound, it follows that $p(m,v,w) \geq p_{\Z^5, D}(m,v_2,w_2) \approx p_{\Z^5}(m,v_2,w_2)$ as before, where $v_2, w_2$ are closet points in $\Z^5$ to $v, w$. However, this does not give us a matching lower bound (as it gives $m^{3/2}$). For the $\Z^4$ page, 
\begin{align*}
p(m,v,w) &\geq p_{\Z^4, D}(m,v_1,w_1) = h(v_1)h(w_1) p_{\Z^4_h}(m,v_1,w_1) \\&\approx h(v_1) h(w_1) \frac{C}{h(v_{1, \sqrt{m}}) h(w_{1, \sqrt{m}}) m^{2}}\exp\Big(-\frac{d^2(v_1,w_1)}{cm}\Big).
\end{align*}
Here $h$ is the $h$-transform for the $\Z^4$ page with the parabola as the boundary, $v_1, w_1$ are closest points in $\Z^4$ to $v,w$, and $v_{1, \sqrt{m}}$ is a point that is approximately distance $\sqrt{m}$ away from both $v_1$ and the parabola. (See \cite[Theorem 3.7]{ed_lsc_trans} and surrounding commentary and references.) 

For $v, w$ fixed and large $m$, this gives us a lower bound of the same order: $m^{-2}$. However, while we describe the behavior of $h$ in some regimes in Section 3.5 of \cite{ed_lsc_trans}, it behaves differently in different regimes. In particular, $h$ does not approach $1$ in any uniform manner as $v,w$ move away from the parabola, and we do not know how $h$ behaves for points near the boundary of the parabola. 
\end{example}

\begin{example}[$\Z^3$ with a half-line tail]\label{recurrent_ex}
Consider a copy of $\Z^3$ and a copy of the half-line $\Z_{\geq 0}$, both with lazy SRW. Identifying the origins $o$ of these graphs creates book-like graph with Harnack and uniform pages. However, the $S$-transient hypothesis does not hold. The upper bound given by Faber-Krahn is 
\[ p(n, o, o) \leq \frac{c_1}{n^{1/2}}.\]
As mentioned previously, Faber-Krahn upper estimates are not sharp in the recurrent (parabolic) case. However, in the case of this example, we can obtain a sharp estimate by making use of Doob's $h$-transform. The continuous version of this example (a copy of $\R^3$ and $\R$ appropriately glued together as manifolds) is treated in Example 6.11 of \cite{lsc_ag_ends}. There a complete description of heat kernel estimates is given; we defer the complete analogous description of the discrete case to a future paper and focus only on the spine to spine estimate.  

First, we construct a positive harmonic function on this glued graph as follows: Consider the Green function $G_{\Z^3}(x,o) = \sum_n p_{\Z^3}(n, x,o)$, where $o$ is the origin. This function is harmonic on $\Z^3 \setminus \{o\}$ (superharmonic on $\Z^3$) and has a finite value at the origin. Indeed, if we set $d(x,o) = |x|,$ then
\[ G_{\Z^3}(x,o) \approx \begin{cases}\frac{1}{|x|}, & x \not = o \\ \sum_{n} \frac{1}{n^{3/2}}, & x=o. \end{cases}\]
Therefore $1+G(x,o)$ is also harmonic on $\Z^3 \setminus \{o\},$ and it approaches $1$ as $|x| \to \infty.$ On the half line, all harmonic functions are linear and have form $ax+b$. Choose $a>0$ such that 
\[ h(x) = \begin{cases} 1+ G_{\Z^3}(x,o), & x \in \Z^3 \setminus \{o\} \\ 1+ G_{\Z^3}(o,o) , & x =o \\ ax + 1+ G_{\Z^3}(o,o), & x \in \Z_{>0} \end{cases}\]
is a harmonic function on $\Z^3$ with a tail. Such an $a>0$ exists since $G_{\Z^3}$ is superharmonic, so the average of points $G_{\Z^3}(x,o)$ at the neighbors of the origin $o$ in $\Z^3$ is less than the value $G_{\Z^3}(o,o).$ 

If $(\Gamma, \mathcal{K}, \pi)$ denotes this glued version of $\Z^3$ and $\Z_{\geq 0}$ with the lazy simple random walk, we consider the $h$-transform graph $(\Gamma_h, \pi_h, \mu_h)$ by setting $\pi_h(x) := h^2(x)\pi(x)$ for all $x \in \Gamma=\Gamma_h$ (as graphs) and 
\[ \mathcal{K}_h(x,y) = \frac{1}{h(x)} \mathcal{K}(x,y) h(y).\]
Then 
\[ p_\Gamma(n,x,y) = h(x)h(y) p_{\Gamma_h}(n,x,y).\]
It follows from the theory of $h$-transforms that if $\Gamma$ has Harnack pages, then the same is true of $\Gamma_h,$ and changing the weights does not change (inner) uniformity. In this example, it is clear from Theorem 2.9 of \cite{ed_lsc_trans} that the pages of $\Gamma_h$ are uniformly $S$-transient, since volume in both pages now behaves like volume in $\Z^3.$ Treating $h(o)$ as a constant, applying the results of Section \ref{spine_hk} to $p_{\Gamma_h}$ gives the sharp bound 
\[ p_\Gamma(n,o,o) = h(o)h(o) p_{\Gamma_h}(n,o,o) \approx \frac{1}{n^{3/2}}.\]
This sharp bound is much smaller than the original upper bound of $n^{-1/2}$.
\end{example}

\begin{example}[$\Z^3$ glued to a $\Z^2$ or any recurrent graph] 
The previous example works similarly if we take a copy of $\Z^3$ and a copy of $\Z^2$, both with lazy SRW, and identify their origins $o$. The Faber-Krahn upper bound gives $p(n,o,o) \leq c_1 n^{-1}.$ 

However, we could again construct a positive harmonic function on the glued graph. Since $\Z^2$ is recurrent, it does not have a Green function, but it does have a function that is harmonic on $\Z^2\setminus \{o\}$. One way to define this function is as $g(x) = \sum_{n=0}^\infty [p_{\Z^2}(n,o,o)-p_{\Z^2}(n,o,x)]$ for $x \in \Z^2.$ It is known that $g(x) \approx \log(|x|),$ where $|x|$ denote the distance from $x$ to the origin $o$. (See e.g. Spitzer \cite{Spitzer} or Lawler and Limic \cite[Chapter 6]{lawler_limic}). Then it is possible to choose $a>0$ such that 
\[ h(x) = \begin{cases} 1+ G_{\Z^3}(x,o), & x \in \Z^3 \setminus \{o\} \\ 1+ G_{\Z^3}(o,o), & x =o \\ ag(x)+1+G_{\Z^3}(o,o), & x \in \Z^2 \setminus \{o\},\end{cases}\]
is a harmonic function on the glued graph. 

Then, as in the previous example, we may do the $h$-transform, and it  again follows from Theorem 2.9 of \cite{ed_lsc_trans} and the new volume of $\Z^2$ after $h$-transform that the origin is uniformly $S$-transient in the $\Z^2$ page. This gives the sharp bound $p(n,o,o) \approx (n \log^2(1+n))^{-1}.$ 

More generally, suppose $\Z^3$ is glued to a recurrent graph $R$ via a single vertex $o$. Assume further that $R$ is a Harnack graph and satisfies the property that there exists $A>1$ such that for any $r>A^2$ and all $x,y \in R$ satisfying $d(o,x)=d(o,y)=r$, there exists a continuous path from $x$ to $y$ whose image is contained in $B(o,Ar) \setminus B(o, A^{-1} r).$ This property is called the \emph{relatively connected annuli} (RCA) property. By work of Griogr'yan and Saloff-Coste \cite[Section 4]{ag_lsc_extcptset}, such a graph $R$ has a function $h_R(x)$ that is harmonic on $R \setminus {o}$ and satisfies 
\[h_R(x) \approx \int_1^{|x|} \frac{s}{V_R(o, s)} \, ds.\]
(While the results of \cite{ag_lsc_extcptset} are about manifolds, the same arguments should hold in the present graph setting). 

Then we may again repeat the construction of a harmonic function on the glued graph and perform the $h$-transform. However, for choices of $R$ that are not $\Z_{>0}$ or $\Z^2,$ we cannot so directly compute $h$, and our heat kernel estimates depend upon this unknown function $h$ as well as upon the volume function of $R$.
\end{example}

\begin{remark}
In all of our examples above, we were always discussing lattices with lazy simple random walk. However, it should be noted that none of our proof techniques rely on having such precise symmetry. All of the conclusions of the examples remain true when gluing graphs quasi-isometric to the $d$-dimensional lattices and with weights comparable to those of the lazy simple random walk.
\end{remark}

\appendix

\section{Proof of Lemma \ref{FK_quasiiso_lem}}\label{FK_quasi_iso}

This appendix reminds the reader of some useful properties of quasi-isometries and provides the proof of Lemma \ref{FK_quasiiso_lem}.

\begin{remark}\label{quasi-iso_compweights}~
\begin{itemize}
\item Suppose $(\Gamma_1, \pi_1, \mu^1)$ and $(\Gamma_2, \pi_2, \mu^2)$ are connected graphs with controlled and uniformly lazy weights, as has been assumed throughout this paper. Suppose further that $\Gamma_1, \Gamma_2$ are quasi-isometric as in Definition \ref{quasi-iso} with a quasi-isometry $\Phi: \Gamma_1 \to \Gamma_2.$ Hypothesis (3), which says the quasi-isometry takes vertices to vertices with comparable weights, can also be thought of as a condition on the volume of small balls. In particular, hypothesis (3) combined with the consequences of controlled and uniformly lazy weights implies
\[ V_1(x,1) \approx V_2(\Phi(x), 1) \text{ and } \mu^1_{xy} \approx \mu^2_{\Phi(x) z} \quad \forall y \sim x \in \Gamma_1, \ z \sim \Phi(x) \in \Gamma_2.\]

\item Suppose $\Gamma_1, \ \Gamma_2$ are quasi-isometric and there is a quasi-isometry $\Phi: \Gamma_1 \to \Gamma_2.$ Then we may always define a quasi-isometry $\Phi^{-1}: \Gamma_2 \to \Gamma_1$ as follows: for every $z \in \Gamma_2,$ set $\Phi^{-1}(z)$  to be an element $\beta \in \Gamma_1$ such that $d_2(\Phi(\beta), z) \leq \varepsilon.$ (Such a $\beta$ always exists as $\Phi$ is a quasi-isometry.) While $\Phi^{-1}$ is not literally the inverse of $\Phi,$ for any such a choice of $\Phi^{-1}$,
\[ d_2(z, \Phi(\Phi^{-1}(z))) = d_2(z, \Phi(\beta)) \leq \varepsilon.\]
Note the quasi-isometry constants of $\Phi^{-1}$ need not be the same as those for $\Phi$ (though they are related).
\end{itemize}
\end{remark}

Many of the arguments and ideas in the following proof of Lemma \ref{FK_quasiiso_lem} resemble those of \cite{TC_LSC_IsoInfini}, which addresses stability of several functional inequalities under quasi-isometry in the context of both Riemannian manifolds and graphs. However, Faber-Krahn inequalities are not discussed there. 

\begin{proof}[Proof of Lemma \ref{FK_quasiiso_lem}]
Let the quasi-isometry constants of $\Phi$ be $a,b,\varepsilon, C_q, C_w$ as in Definition \ref{quasi-iso}. In this proof, the Latin alphabet denotes elements of $\widehat{\Gamma},$ and the Greek alphabet denotes elements of $\Gamma.$ 

Let $r >0,\ z \in \widehat{\Gamma}.$ Consider $\widehat{\Omega} \subset \widehat{B}(z, r)$ and fix a function $f$ supported on $\widehat{\Omega}.$ Since the eigenfunction of $\lambda_1(\Omega)$ has a sign, we may assume $f \geq 0.$ 

We will also assume $r \geq 1:$ If $r<1$ then $\widehat{B}(z, r) = \{z\}$ and the only choice is $\widehat{\Omega} = \widehat{B}$ so $f(z) =C$ and $f(x) = 0$ for $x \in \widehat{\Gamma}$ where $x \not = z.$ Then, if $\widehat{C}_c$ is the constant for controlled weights in $\widehat{\Gamma},$
\begin{align*}
&\lVert \, \lvert \nabla f \rvert \, \rVert_2^2 = C^2 \sum_{y \sim z} \widehat{\mu}_{zy} \geq C^2 \frac{\widehat{\pi}(z)}{\widehat{C}_c} = \frac{1}{\widehat{C}_c} \lVert f \rVert^2.
\end{align*}
Thus $\lambda_1(\widehat{\Omega}) \approx 1.$ This estimate is independent of $z,$ and similarly, $\lambda_1(\{\alpha\}) \approx 1$ for any $\alpha \in \Gamma.$ Thus these Faber-Krahn functions evaluated at such balls are always similar and we can focus on the more interesting case where $r \geq 1.$

Let $f_\varepsilon$ denote the weighted average of $f$ on balls of radius $\varepsilon,$ that is, 
\[ f_\varepsilon(x) = \frac{1}{\widehat{V}(x,\varepsilon)} \sum_{y \in \widehat{B}(x,\varepsilon)} f(y) \, \widehat{\pi}(y).\]

\noindent\underline{Support of $f_{\varepsilon} \circ \Phi$:}

Consider $f_\varepsilon \circ \Phi,$ which is a function on $\Gamma.$

By definition, $\supp f \subseteq \widehat{\Omega} \subseteq \widehat{B}(z,r).$ Then $\supp f_\varepsilon \subseteq [\widehat{\Omega}]_{\varepsilon} \subseteq \widehat{B}(z, r + \varepsilon) \subseteq \widehat{B}(z, (1+\varepsilon)r),$ where $[ \widehat{\Omega} ]_{\varepsilon}$ denotes the $\varepsilon$-neighborhood of $\widehat{\Omega}.$ 

We claim that if $\alpha \in \supp (f_\varepsilon \circ \Phi),$ then $\alpha \in B(\Phi^{-1}(z), c_2 r),$ where $c_2 : = a(1 + 2 \varepsilon+ b).$ 

Suppose $\alpha \in \supp (f_\varepsilon \circ \Phi).$ Then $\Phi(\alpha) \in \widehat{B}(z, (1+\varepsilon)r)$ and hence: 
\begin{align*}
d(\alpha, \Phi^{-1}(z)) &\leq a \widehat{d}(\Phi \circ \Phi^{-1} (z), \Phi(\alpha)) + ab
\leq a \Big[ \widehat{d}(\Phi \circ \Phi^{-1}(z), z) + \widehat{d}(z, \Phi(\alpha))\Big] +ab \\
&\leq a \big[\varepsilon + (1+\varepsilon)r \big] +ab
\leq a(1 + 2 \varepsilon+ b) r. 
\end{align*}

Moreover, let $\Omega := \text{PreIm}_{\Phi}([\widehat{\Omega}]_{\varepsilon}) \subset B(\Phi^{-1}(z), c_2r).$ Note $\alpha \in \Omega \iff \Phi(\alpha) \in [\widehat{\Omega}]_{\varepsilon}.$ Hence $\supp (f_\varepsilon \circ \Phi) \subset \Omega \subset B(\Phi^{-1}(z), c_2r).$ 

\noindent\underline{Using a Faber-Krahn Function on $\Gamma$:}

Since $\Lambda$ is a relative Faber-Krahn function on $\Gamma,$ we have 
\begin{align}
\lVert \, | \nabla  (f_\varepsilon \circ \Phi) |\, \rVert_2^2 \geq \Lambda(B(\Phi^{-1}(z), c_2r), \pi(\Omega))\, \lVert f_\varepsilon \circ \Phi||_2^2.
\end{align}

The result will follow if we can prove the following three claims:
\begin{enumerate}
\item $\pi(\Omega) \leq c_3\, \widehat{\pi}(\widehat{\Omega})$ for some constant $c_3$ (since $\Lambda$ is decreasing in $\nu$ by definition)
\item $\lVert \, | \nabla ( f_\varepsilon \circ \Phi) | \, \rVert_2^2 \leq C_A \lVert \, | \nabla  f|\, \rVert_2^2$ for some constant $C_A$
\item $\lVert f_\varepsilon \circ \Phi||_2^2 \geq C_B \lVert f \rVert_2^2$ for some constant $C_B$. 
\end{enumerate}

(Note above we must be careful to choose the appropriate $L^2$ norms with respect to either $\Gamma$ or $\widehat{\Gamma}$.)

\underline{Claim 1:} $\pi(\Omega) \leq c_3\, \widehat{\pi}(\Omega)$ for some constant $c_3.$

As $\widehat{\Gamma}$ is locally uniformly finite, each point in $\widehat{\Gamma}$ has at most $\widehat{N}$ neighbors. Further, suppose $w \in [\widehat{\Omega}]_{\varepsilon} \setminus \widehat{\Omega}.$ Then there exists some point $x \in \widehat{\Omega}$ such that $\widehat{d}(x,w) < \varepsilon$ and hence $\widehat{\pi}(x) \approx \widehat{\pi}(w)$ (with constants independent of $x,\ w$). The contribution of $w$ to $\widehat{\pi}([\widehat{\Omega}]_{\varepsilon})$ is $\widehat{\pi}(w).$ For each point $x \in \widehat{\Omega},$ there are at most finitely many $w$'s in $[\widehat{\Omega}]_{\varepsilon} \setminus \widehat{\Omega}$ satisfying $\widehat{d}(x,w) < \varepsilon.$ Hence there is some constant $C$ (depending on $\varepsilon$ and the constants of the hypotheses on the graphs) such that 
\begin{align*}
\widehat{\pi}(\widehat{\Omega}) \leq \widehat{\pi}([\widehat{\Omega}]_{\varepsilon}) \leq C \, \widehat{\pi}(\widehat{\Omega}).
\end{align*}

By definition, $\Omega = \text{PreIm}_{\Phi}([\widehat{\Omega}]_{\varepsilon}).$ Each point in $[\widehat{\Omega}]_{\varepsilon}$ can only be mapped to by a finite number of points in $\Omega$ since $\Gamma_1$ is locally uniformly finite and
\[ a d(\alpha, \beta) - b \leq \widehat{d}(\Phi(\alpha), \Phi(\beta)) = 0 \implies d(\alpha, \beta) \leq \frac{b}{a}.\]

Then, where $c$ is a constant (depending on the graph structure and the quasi-isometry) whose value changes in each step, the above and earlier remarks imply
\begin{align*}
\pi(\Omega) &= \sum_{\alpha \in \Omega} \pi(\alpha) \leq c \sum_{\alpha \in \Omega} \widehat{\pi}(\Phi(\alpha)) 
\leq c \sum_{v \in [\widehat{\Omega}]_{\varepsilon}} \widehat{\pi}(v)
\leq c \widehat{\pi}(\widehat{\Omega}).
\end{align*}

We may take $c_3$ to be any value of $c$ that achieves this upper bound.

\underline{Claim 2:} $\lVert \, | \nabla  (f_\varepsilon \circ \Phi) | \, \rVert_2^2 \leq C_A \lVert \,  | \nabla  f| \, \rVert_2^2$ for some constant $C_A$

Let $\alpha, \beta \in \Gamma$ satisfy $\alpha \sim \beta.$ Then $\widehat{d}(\Phi(\alpha), \Phi(\beta)) \leq a d(\alpha, \beta) = a.$ Let $x = \Phi(\alpha), y = \Phi(\beta);$ by the previous sentences, $y \in \widehat{B}(x, a).$ Consequently, for $w \in \widehat{B}(y, \varepsilon),$ we have $w \in \widehat{B}(x, \varepsilon + a).$ Hence
\begin{align*}
|f_{\varepsilon}(x) - f_{\varepsilon}(y)| &= |f_{\varepsilon}(x) - f(x) + f(x) - f_{\varepsilon}(y)| \\
&\leq \frac{1}{|\widehat{V}(x, \varepsilon)|} \sum_{v \in \widehat{B}(x,\varepsilon)} |f(v) - f(x)| \, \widehat{\pi}(v)  + \frac{1}{|\widehat{V}(y, \varepsilon)|} \sum_{w \in \widehat{B}(y, \varepsilon)} |f(x) - f(w)| \, \widehat{\pi}(w)  \\
&\leq \frac{C}{|\widehat{V}(x, \varepsilon+a)|} \sum_{v \in \widehat{B}(x, \varepsilon + a)} |f(v) - f(x)| \, \widehat{\pi}(v) .
\end{align*} 

Using Jensen's inequality, 
\begin{align*}
|f_{\varepsilon}(x) - f_{\varepsilon}(y)|^2 &\leq \Big|\frac{C}{|\widehat{V}(x, \varepsilon+a)|} \sum_{v \in \widehat{B}(x, \varepsilon + a)} |f(v) - f(x)| \; \widehat{\pi}(v)\Big|^2 \\
&\leq \frac{C}{|\widehat{V}(x, \varepsilon+a)|} \sum_{v \in \widehat{B}(x, \varepsilon + a)} |f(v) - f(x)|^2 \; \widehat{\pi}(v)
\end{align*}

Now, for each $v \in \widehat{B}(x, \varepsilon + a),$ take a path of vertices $x_0 = x, x_1, \dots, x_{\widehat{d}(x,v)} = v$ between $x$ and $v$ such that $\widehat{d}(x_i, x_{i+1}) = 1$ for all $i.$ Then since such a path has a bounded length,
\[ |f(v) - f(x)|^2 \leq \Big| \sum_{i=0}^{\widehat{d}(x,v) - 1} |f(x_i) - f(x_{i+1})| \Big|^2 \leq C\sum_{i=0}^{\widehat{d}(x,v) - 1} |f(x_i) - f(x_{i+1})|^2 .\]

Moreover $\widehat{\pi}(v) \approx \widehat{\pi}(x_i)$ for any $i=0, \dots, \widehat{d}(x,v) - 1$ since $\widehat{d}(x_i, v) \leq \varepsilon +a,$ a fixed value. Putting the above estimates together,
\[|f_{\varepsilon}(x) - f_{\varepsilon}(y)|^2  
\leq  \frac{C}{|\widehat{V}(x, \varepsilon+a)|} \sum_{v \in \widehat{B}(x, \varepsilon + a)} \sum_{i=0}^{d(x,v) - 1} |f(x_i) - f(x_{i+1})|^2 \; \widehat{\pi}(x_i). \]

Recall $x= \Phi(\alpha)$ (and $y=\Phi(\beta)).$ Using Remark \ref{quasi-iso_compweights} and the controlled weights of the graphs,
\[ \mu_{\alpha \beta} \approx \widehat{\mu}_{\Phi(\alpha) x_1} \approx \widehat{\pi}(\Phi(\alpha)) \approx \widehat{\pi}(x_i) \approx \widehat{\mu}_{x_i x_{i+1}}.\]

Now, as at most finitely points of $\Gamma$ can map to the same point in $\widehat{\Gamma},$
\begin{align*}
\lVert \, | \nabla  (f_\varepsilon \circ \Phi) |\,  \rVert_2^2 &= \frac{1}{2} \sum_{\alpha, \beta \in \Gamma} |f_{\varepsilon} \circ \Phi (\alpha) - f_{\varepsilon} \circ \Phi (\beta)|^2 \mu_{\alpha \beta} 
= \frac{1}{2} \sum_{\alpha, \beta \in \Gamma} |f_{\varepsilon}(x) - f_{\varepsilon} (y)|^2 \mu_{\alpha \beta} \\
&\leq C \sum_{\alpha \in \Gamma} \frac{1}{|\widehat{V}(x, \varepsilon+a)|} \sum_{v \in \widehat{B}(x, \varepsilon + a)} \sum_{i=0}^{d(x,v) - 1} |f(x_i) - f(x_{i+1})|^2 \, \widehat{\pi}(x_i) \, \widehat{\mu}(x_i, x_{i+1}) \\
&\leq C \sum_{w \in \widehat{\Gamma}} \frac{1}{|\widehat{V}(w, \varepsilon+a)|} \sum_{v \in \widehat{B}(w, \varepsilon + a)} \sum_{i=0}^{d(w,v) - 1} |f(x_i) - f(x_{i+1})|^2 \, \widehat{\pi}(x_i) \, \widehat{\mu}(x_i, x_{i+1}) \\
&\leq  C \sum_{w \in \widehat{\Gamma}}  \sum_{v \in \widehat{B}(w, \varepsilon + a)} \sum_{i=0}^{d(w,v) - 1} |f(x_i) - f(x_{i+1})|^2 \, \widehat{\mu}(x_i, x_{i+1}),
\end{align*}
where in the last line we used that $\widehat{\pi}(x_i) \approx \widehat{\pi}(w) \approx \widehat{V}(w, \varepsilon+a)$ for all $x_i \in \widehat{B}(w, \varepsilon +a).$

Further, suppose $u_1, u_2 \in \widehat{\Gamma}$ and $u_1 \sim u_2.$ Then there are a finite number of elements $w \in \widehat{\Gamma}$ such that $u_1, u_2 \in \widehat{B}(w, \varepsilon + a).$ Further, in each such ball, there are also a finite number of elements $v$ where $u_1, u_2$ could appear on a shortest path between $w$ and $v.$ Hence in the above sum, each pair $u_1, u_2$ appears at most finitely many times. Thus
\begin{align*}
\lVert \, | \nabla  (f_\varepsilon \circ \Phi) | \, \rVert_2^2
&\leq C \sum_{u_1 \in \widehat{\Gamma}} \sum_{u_2 \sim u_1} |f(u_1) - f(u_2)|^2 \; \widehat{\mu}_{u_1 u_2} = C_A \lVert \, | \nabla f | \, \rVert_2^2.
\end{align*}

\underline{Claim 3:} $\lVert f_\varepsilon \circ \Phi||_2^2 \geq C_B \lVert f \rVert_2^2$ for some constant $C_B$.

By our earlier arguments about the support of $f_\varepsilon \circ \Phi,$ we have
\begin{align*}
\lVert f_{\varepsilon} \circ \Phi \rVert_2^2 = \sum_{\alpha \in \Gamma} |f_\varepsilon \circ \Phi(\alpha)|^2 \, \pi(\alpha)
= \sum_{\alpha \in B(\Phi^{-1}(z), c_2 r)} |f_\varepsilon \circ \Phi(\alpha)|^2 \, \pi(\alpha).
\end{align*}

\noindent Take a set of points $\{w_i\}_{i=1}^{i_\varepsilon} \subset \widehat{\Gamma}$ such that 
\begin{itemize}
\item $w_i = \Phi(\alpha_i)$ for some $\alpha_i \in \Gamma$ (as the $w_i$ are distinct, so are the $\alpha_i$)
\item $\widehat{B}(z,r) \cap \widehat{B}(w_i, \varepsilon) \not = \emptyset$ 
\item $\widehat{B}(z,r) \subseteq \bigcup_{i=1}^{i_\varepsilon} \widehat{B}(w_i, \varepsilon)$.
\end{itemize}
Such a set of points exists since every point in $x \in \widehat{\Gamma}$ must be at distance at most $\varepsilon$ from a point $w = \Phi(\alpha)$. 

This choice forces each $\alpha_i \in B(\Phi^{-1}(z), c_2  r),$ so that all of the $w_i$'s appear in the $L^2$ norm of $f_{\varepsilon} \circ \Phi.$ Hence, recalling $\pi(\alpha) \approx \widehat{\pi}(w_i),$ 
\begin{align*}
\lVert f_{\varepsilon} \circ \Phi \rVert_2^2
& \geq c \sum_{i=1}^{i_\varepsilon} |f_{\varepsilon}(w_i)|^2 \, \widehat{\pi}(w_i).
\end{align*}

For each $x \in \widehat{B}(z,r),$ there exists at least one $j$ such that $x \in \widehat{B}(w_j, \varepsilon).$ Thus
\begin{align*}
\frac{f(x)\, \widehat{\pi}(x)}{|\widehat{V}(w_j, \varepsilon)|} \leq \frac{1}{|\widehat{V}(w_j, \varepsilon)|} \sum_{y \in \widehat{B}(w_j, \varepsilon)} f(y) \widehat{\pi}(y) = f_{\varepsilon}(w_j).
\end{align*}

Up to a constant, the left hand side above is just $f(x),$ since $\widehat{\pi}(x) \approx \widehat{\pi}(w_j) \approx \widehat{V}(w_j, \varepsilon).$

As at most finitely many elements $x \in \widehat{B}(z,r)$ are in the same $\widehat{B}(w_j, \varepsilon)$ (and, in fact, we have a uniform bound on how many such elements there are), and these small balls cover the entire larger ball, we find
\begin{align*}
\lVert f_{\varepsilon} \circ \Phi \rVert_2^2
& \geq c \sum_{i=1}^{i_\varepsilon} |f_{\varepsilon}(w_i)|^2 \, \widehat{\pi}(w_i)
\geq c \sum_{x \in \Omega} |f(x)|^2 \; \widehat{\pi}(x) = C_B \lVert f \rVert_2^2.
\end{align*}
\end{proof}

\section{Proof of Theorem \ref{FK_implies_HK}}\label{FK_implies_HK_app}

There are some very general results about heat kernel upper bounds and Faber-Krahn functions in the continuous case, especially in the work of Grigor'yan; see for example Theorem 5.2 of \cite{ag_upperbds}. While there are some such results in the discrete case, they are not as general. In \cite{TC_AG_GraphVol}, Coulhon and Grigor'yan prove the equivalence of $\Gamma$ being Harnack and having relative Faber-Krahn function of form (\ref{Harnack_FK}). In particular, this implies that if $\Gamma$ has a relative Faber-Krahn function of form (\ref{Harnack_FK}), then $p(n,x,y)$ satisfies the upper bound
\[ p(n,x,y) \leq \frac{C}{\sqrt{V(x, \sqrt{n}) V(y, \sqrt{n})}} \exp\Big(-\frac{d_\Gamma^2(x,y)}{c n}\Big).\]

We want a generalization of this result that says if $\Gamma$ satisfies (\ref{FK_withF}), then $p(n,x,y)$ has the upper bound given by (\ref{HK_upper_FK}); this amounts to replacing $V_\Gamma(z,r)$ by the function $F(z,r).$ It is not possible to simply replace $V$ with $F$ in the entire argument of \cite{TC_AG_GraphVol}, in part because having a relative Faber-Krahn inequality of form (\ref{Harnack_FK}) implies $\Gamma$ is doubling, which need not be true in the case considered here. Nonetheless, we still follow the same general approach of \cite{TC_AG_GraphVol, ag_upperbds}. We begin by collecting together some lemmas.

We start with an $L^2$ mean value inequality as in Section 4 of \cite{TC_AG_GraphVol}.

\begin{lemma}\label{L_2_toL_infinity}
Let $(\Gamma, \mu, \pi)$ be a graph with a relative Faber-Krahn function of the form (\ref{FK_withF}) and let $u(T,z)$ be a non-negative sub-solution of the discrete heat equation on $\mathbb{N} \times B(z,R)$ as in Definition \ref{subsoln}. Then 
\begin{equation}
u^2(T,z) \leq \frac{C}{F(z,R) \min\{T^{1/\alpha+1}R^{-2/\alpha}, R^{2}\}} \sum_{k=0}^{2T} \sum_{x \in B(z,R)} u^2(k,x) \pi(x)
\end{equation}
for all $T, R >0$ and $z \in B(x,R).$ 
\end{lemma}

\begin{proof}
We can essentially copy the proof given in \cite{TC_AG_GraphVol}, Sections 4.4 and 4.5. Mostly, they treat 
\[ \beta = \frac{a}{R^2}V(z,R)^\alpha\]
as a constant. All of the differences enter into the quantity $\beta,$ as we have instead 
\[ \widetilde{\beta} = \frac{a}{R^2} F(z,R)^\alpha.\] 

The only time the volume is explicitly used in \cite{TC_AG_GraphVol} is to prove that 
\begin{equation}\label{magic_beta}
 \beta \leq c \pi(z)^\alpha \quad \implies \beta^{\alpha^{-1}\gamma^{-N}} \leq c \pi(z)^{\gamma^{-N}},
 \end{equation}
an inequality which ``magically'' makes certain factors in the somewhat involved iteration proof cancel. However, inequality (\ref{magic_beta}) holds for $\widetilde{\beta}$ for essentially the same reasons. The inequality will follow once we show
\[ \frac{F(z,R)}{\pi(z)} \leq c R^{2/\alpha}.\]

Consider $B=B(z,R)$ and let $\Omega = \{z\},$ which is a subset of $B(z,R).$ Let $f$ be the function on $\Gamma$ that is $1$ at $z$ and zero everywhere else. Therefore, by the form (\ref{FK_withF}) of the Faber-Krahn function on $\Gamma,$
\begin{align*}
&\frac{1}{2} \sum_{x,y \in \Gamma} |f(y)-f(x)|^2 \mu_{xy} \geq \Lambda(B, \pi_z) \sum_{x \in B(z,R)} |f(x)|^2 \pi_x \\
&\implies \sum_{y \sim z} \mu_{yz} \geq \frac{a}{R^2}\Big(\frac{F(z,R)}{\pi_z}\Big)^\alpha \pi_z.
\end{align*} 
Since the weights are subordinated to the measure, it follows that $\sum_{y\sim z} \mu_{yz} \leq \pi_z.$ Therefore we get precisely the desired inequality 
\[ \Big(\frac{F(z,R)}{\pi_z}\Big)^\alpha \leq C R^2.\]

Other than this modification, the proof goes through exactly as in \cite{TC_AG_GraphVol}. 
\end{proof}

Lemma \ref{L_2_toL_infinity} can be considered an $L^2 \to L^\infty$ type estimate. We want to obtain an $L^1$ version of this result. There are several possible ways to proceed, but the approach we take here is to try to directly repeat the arguments given in the continuous setting in \cite{ag_upperbds}. One of the main obstacles present in the discrete work \cite{TC_AG_GraphVol} is that at the time of that paper's publication, there was not yet a discrete version of the integrated maximum principle, which is a key component of the arguments of \cite{ag_upperbds}. Hence the methods in \cite{TC_AG_GraphVol} necessarily avoid appealing to this principle. However, Coulhon, Grigor'yan, and Zucca give a discrete version of the integrated maximum principle in \cite{TC_AG_Zucca_MaxPrin}. We make use of several of their results to follow the method of \cite{ag_upperbds}. 

\begin{lemma}[Discrete integrated maximum principle, Theorem 2.2 \cite{TC_AG_Zucca_MaxPrin}]\label{disc_max_prin}~
Let $(\Gamma, \mu, \pi)$ be connected graph with controlled and uniformly lazy weights, where $\alpha$ is a constant such that $\mathcal{K}(x,x) \geq \alpha$ for all $x \in \Gamma.$ Let $f$ be a strictly positive function on $[0, n] \times \Gamma$ such that for all $x \in \Gamma,\ k \in [0,n)$, 
\begin{equation}\label{disc_max_prin_cond}
\partial_k f(x) + \frac{|\nabla f_{k+1}|^2}{4 \alpha f_{k+1}}(x) \leq 0. 
\end{equation}
Then, for any solution $u$ of the heat equation in $[0,n) \times \Gamma,$ the quantity 
\begin{equation}
J_k = J_k(u) := \sum_{x \in \Gamma} u_k^2(x) f_k(x) \pi(x)
\end{equation}
is non-increasing in $k,$ that is, $J_{k+1} \leq J_k$ for all $k \in [0, n).$ 
\end{lemma}

\begin{lemma}[Proposition 2.5 of \cite{TC_AG_Zucca_MaxPrin}]\label{good_max_fcn}
Let $(\Gamma, \mu, \pi)$ be a connected graph with controlled and uniformly lazy weights. Let $\rho$ be a $1$-Lipschitz function on $\Gamma$ such that $\inf \rho \geq 1.$ Then there exists a positive number $D_\alpha$ (depending only on $\alpha$ from the uniformly lazy condition) such that for all $D \geq D_\alpha,$ the weight function 
\begin{equation}
f_k(x) = f_k^D(x) := \exp \Big( -\frac{\rho^2(x)}{D(n+1-k)}\Big)
\end{equation}
satisfies (\ref{disc_max_prin_cond}) for all $x \in \Gamma, \ k \in [0,n).$ 
\end{lemma}

\begin{lemma}[Proposition 5.3 of \cite{TC_AG_Zucca_MaxPrin}]\label{E_D_lemma} 
Again let $(\Gamma,\mu,\pi)$ be a connected graph with controlled and uniformly lazy weights. Define the quantity 
\begin{equation}
E_D(k,x) := \sum_{z \in \Gamma} p^2(k,x,z) \exp \Big( \frac{\rho^2(x,z)}{Dk}\Big) \, \pi(z),
\end{equation}
where $\rho(x,z) := \max \{ d(x,z) , 1 \}.$ 

Then for all $x,y \in \Gamma, \ k \in \N,$ and all $D>0,$ 
\begin{equation}
p(2k, x,y) \leq \sqrt{E_D(k,x) E_D(k,y)} \exp \Big( - \frac{d^2(x,y)}{4Dk} \Big).
\end{equation}
\end{lemma}

We now use the above lemmas to prove a discrete version of Theorem 4.1 of \cite{ag_upperbds}. The proof is essentially the same and simply requires appealing to the above discrete setting lemmas (instead of their continuous analogs). 

\begin{lemma}\label{est_with_ball} 
As above, let $(\Gamma, \mu, \pi)$ be a connected graph with controlled and uniformly lazy weights with relative Faber-Krahn function of form (\ref{FK_withF}). Let $B=B(z,R), \  \rho(x,y) = \max\{1, d(x,y)\},$ and $\rho(x, B) = \max \{ 1 , d(x, B(z,R)) \} = \max \{ 1, (d(x,z) - R)_+\}$. Then 
\begin{align*}
\sum_{y \in \Gamma} p(k,y,z)^2 \exp\Big(\frac{\rho(y,B)^2}{\hat{c}(T+1)}\Big)\, \pi(y) &\leq \frac{\tilde{c}\, T}{F(z,R) \min \{ R^2, T^{1+1/\alpha} R^{-2/\alpha}\}} \\
&= \frac{c}{F(z,R) \min\{(R^2/T), (T/R^2)^{1/\alpha}\}}.
\end{align*}
\end{lemma}

\begin{proof}
Fix $B = B(z,R).$ Let $\varphi$ be a function on $\Gamma$ with finite support, and take $\Omega \subset \Gamma$ finite and containing both $\supp \varphi$ and $B(z,R).$

Set 
\begin{equation}
u_\Omega(k,x) := \sum_{y \in \Omega} p_{\Omega,D}(k,x,y) \varphi(y) \pi(y).
\end{equation}
By properties of the heat kernel $p, \ u_\Omega$ is a solution of the heat equation on $\N \times B(z,R)$. Without loss of generality, we may assume that $\varphi \geq 0,$ and, consequently, the same is true of $u_\Omega.$  Applying Lemma \ref{L_2_toL_infinity} to $u_\Omega$ yields
\begin{equation}
u_\Omega^2(T,z) \leq \underbrace{\frac{C}{F(z,R) \min\{T^{1/\alpha+1}R^{-2/\alpha}, R^{2}\}}}_{A} \sum_{k=0}^{2T} \, \sum_{x \in B(z,R)} u_\Omega^2(k,x) \pi(x).
\end{equation}
Now consider the function 
\[ f(s,x) := \exp \Big( - \frac{\rho(x,B)^2}{\hat{c}(T+1-s)}\Big).\]
For $\hat{c}$ sufficiently large, $f(s,x)$ is the sort of function considered Lemma \ref{good_max_fcn}. When $x \in B, \ \rho(x,B) = 1$ and $f(x,s) = \exp(-1/(\hat{c}(T+1-s))),$ which is largest when $s$ is smallest. In particular, when $x \in B,\ \exp(-1/\hat{c}) \leq f(s,x) \leq \exp(-1/(\hat{c}(T+1))).$ 
Therefore
\begin{align*}
u_\Omega^2(T,z) &\leq A \sum_{k=0}^{2T} \, \sum_{x \in B(z,R)} u_\Omega^2(k,x) \pi(x) \exp\Big(-\frac{1}{\hat{c}}\Big)\exp\Big(\frac{1}{\hat{c}}\Big) \\
& \leq A \exp\Big(\frac{1}{\hat{c}}\Big) \sum_{k=0}^{2T} \, \sum_{x \in B(z,R)} u_\Omega^2(k,x) f(k,x) \pi(x) \\
&\leq A C(\hat{c}) \sum_{k=0}^{2T} \underbrace{\sum_{x \in \Omega} u_\Omega^2(k,x) f(k,x) \pi(x)}_{B_k}
\end{align*} 
We now want to apply the discrete integrated maximum principle. As mentioned above, $u_\Omega$ solves the heat equation in $\Omega$, and, by choice of $\hat{c},$ we also have that $f$ satisfies (\ref{disc_max_prin_cond}). Therefore, $B_k$ is decreasing in $k,$ which means it is largest when $k=0.$ Consequently
\begin{align*}
u_\Omega^2(T,z) &\leq 2A C(\hat{c}) T \sum_{x \in \Omega} u_\Omega^2(0,x) f(0,x) \pi(x)  \\
&= 2A C(\hat{c}) T \sum_{x \in \Omega} \bigg[ \sum_{y \in \Omega} p_\Omega(0,x,y) \varphi(y) \pi(y) \bigg]^2 \exp\Big(- \frac{\rho(x,B)^2}{\hat{c}(T+1)}\Big)\, \pi(x)  \\\
&= 2A C(\hat{c}) T \sum_{x \in \Omega} \varphi^2(x) \exp\Big(- \frac{\rho(x,B)^2}{\hat{c}(T+1)}\Big) \, \pi(x).
\end{align*}
All sums above were finite and well-defined since $\Omega$ is itself finite. As $\Omega \to \Gamma,$ then $u_\Omega \to u_\Gamma$ and
\begin{align*}
u_\Gamma^2(T,z) = \bigg[\sum_{y \in \Gamma} p(T,z,y)\varphi(y) \pi(y)\bigg]^2 \leq 2A C(\hat{c}) T \sum_{x \in \Gamma} \varphi(x)^2 \exp\Big(- \frac{\rho(x,B)^2}{\hat{c}(T+1)}\Big)\, \pi(x).
\end{align*} 

Now think of $T,z$ as fixed and consider a map $\Phi : L^2(\Gamma) \to \R$ given by 
\begin{equation*}
\Phi(\eta) := \sum_{ y \in \Gamma} p(T,y,z) \exp\Big(\frac{\rho(y, B)^2}{2\hat{c}(T+1)}\Big)\, \eta(y)\, \pi(y). 
\end{equation*}

If $\eta(y)$ has the form $\exp\Big(-\frac{\rho(y, B)^2}{2\hat{c}(T+1)}\Big) \varphi(y),$ then $\Phi(\eta) = u(T,z).$ In this case,
\[ \Phi(\eta)^2 \leq A C(\hat{c}) T \lVert \eta \rVert_{L^2(\Gamma)}^2. \]

Since functions $\eta$ of the form $\exp\Big(-\frac{\rho(y, B)^2}{2\hat{c}(T+1)}\Big) \varphi(y)$ (where $\varphi$ has compact support) are dense in $L^2(\Gamma),$ 
\begin{equation}
\lVert \Phi(\eta) \rVert_\R^2 = |\Phi(\eta)|^2 \leq A C(\hat{c}) T \lVert \eta \rVert_{L^2(\Gamma)}^2, 
\end{equation} 
which imples $\lVert \Phi \rVert^2 \leq A C(\hat{c})T$ and that $\Phi$ is well-defined.

On the other hand, $\lVert \Phi \rVert = \sup \{ \lVert \Phi(\eta) \rVert : \lVert \eta \rVert_{L^2} = 1\},$ so by Cauchy-Schwartz,
\begin{align*}
\lVert \Phi(\eta) \rVert &\leq \bigg[ \sum_{y \in \Gamma} p^2(T,y,z) \exp\Big(\frac{\rho(y,B)^2}{\hat{c}(T+1)}\Big) \pi(y) \bigg]^{1/2} \bigg[ \sum_{y \in \Gamma} \eta^2(y) \pi(y) \bigg]^{1/2} \\
&= \bigg[ \sum_{y \in \Gamma} p^2(T,y,z) \exp\Big(\frac{\rho(y,B)^2}{\hat{c}(T+1)}\Big) \pi(y) \bigg]^{1/2}.
\end{align*}

The supremum occurs when the above inequality is an equality, and hence
\begin{align*}
\lVert \Phi \rVert^2 = \sum_{y \in \Gamma} p^2(T,y,z) \exp\Big(\frac{\rho(y,B)^2}{\hat{c}(T+1)}\Big) \pi(y) \leq A C(\hat{c}) T .
\end{align*}
Recalling the definition of $A$ and simplifying, we get the desired inequality
\begin{align*}
\sum_{y \in \Gamma} p^2(T,y,z) \exp\Big(\frac{\rho(y,B)^2}{\hat{c}(T+1)}\Big) \, \pi(y) \leq \frac{Ce^{-1/\hat{c}}}{F(z,R) \min \{ (T/R^2)^{1/\alpha}, R^2/T\}}.
\end{align*}  
\end{proof}

We are at last ready to prove the theorem. 

\begin{proof}[Proof of Theorem \ref{FK_implies_HK}]

We begin with the result of Lemma \ref{E_D_lemma}: 
\[ p(2k, x,y) \leq \sqrt{E_D(k,x) E_D(k,y)} \exp\Big(-\frac{d^2(x,y)}{4Dk}\Big).\]

Recall
\[ E_D(k,x) = \sum_{w \in \Gamma} p^2(k,x,w) \exp\Big(\frac{\rho(x,w)^2}{Dk}\Big) \, \pi(w),\]
where $\rho(x,w) := \max\{ 1, d(x,w)\}.$ 
The only difference between $E_D$ and the left hand side of Lemma \ref{est_with_ball} is in the exponential term. In order to apply Lemma \ref{est_with_ball} to $E_D,$ we need to justify that we can replace 
$ \exp\Big(\frac{\rho(x,w)^2}{Dk}\Big)$ with $\exp\Big(\frac{\rho(w,B)^2}{\hat{c}(k+1)}\Big),$ where we take $B = B(x,R).$ 

This sort of estimate is only of interest to us if $k\geq 1,$ so $k+1 \approx k.$ As in \cite{ag_upperbds} (see the proof of Corollary 4.1), a quadratic inequality holds. In particular, provided that $D > 2 \hat{c},$
\begin{equation}
\frac{2\hat{c}}{D} \rho(x,y)^2 - \frac{2\hat{c}}{D-2\hat{c}}\, R^2 \leq \rho(w,B)^2.
\end{equation}

Therefore
\begin{align*}
\exp\bigg(\frac{\rho(x,w)^2}{Dk}\bigg) &\leq \exp\bigg(\frac{R^2}{(D-2\hat{c})k}\bigg)\exp\bigg(\frac{\rho(w,B)^2}{2\hat{c}k}\bigg) \\&\leq 
\exp\bigg(\frac{R^2}{(D-2\hat{c})k}\bigg)\exp\bigg(\frac{\rho(w,B)^2}{\hat{c}(k+1)}\bigg).
\end{align*}

Hence, applying Lemma \ref{est_with_ball},
\begin{align*}
E_D(k,x) &\leq \sum_{w \in \Gamma} p^2(k,x,w) \exp\bigg(\frac{R^2}{(D-2\hat{c})k}\bigg)\exp\bigg(\frac{\rho(w,B)^2}{\hat{c}(k+1)}\bigg)\, \pi(w) \\
&\leq \exp\bigg(\frac{R^2}{(D-2\hat{c})k}\bigg) \frac{C(\hat{c})}{F(x,R) \min\{(R^2/k), (k/R^2)^{1/\alpha}\}}.
\end{align*}

Now choose $R=\sqrt{2k};$ then $R^2/k$ is the constant $2$ and 
\[ E_D(k,x) \leq \frac{C(\hat{c}, D) }{F(x,\sqrt{2k})}.\] 
Chaining inequalities above together gives 
\begin{align}\label{final_est}
p(2k,x,y) \leq \frac{C}{\sqrt{F(x, \sqrt{2k}) F(y, \sqrt{2k})}} \exp\Big(-\frac{d^2(x,y)}{\tilde{c}(2k)}\Big),
\end{align}
where $C,\tilde{c}$ depend on $\hat{c}.$ Tracing back through the lemmas, $\hat{c}$ only depends upon the value $\alpha$ from the uniformly lazy condition (i.e. $\mathcal{K}(x,x) \geq \alpha \ \forall x \in \Gamma$), and not on $x,y,$ or $k.$ Inequality (\ref{final_est}) gives the desired result for even times; to get the result for odd times, the difference between one more step can be controlled by constants due to the hypotheses of controlled and uniformly lazy weights. It is also not necessary to take $k$ an integer in the proof above, provided some modifications are made with floor/ceiling functions where appropriate. 
\end{proof}

%




 \providecommand{\bysame}{\leavevmode\hbox to3em{\hrulefill}\thinspace}
 \providecommand{\MR}{\relax\ifhmode\unskip\space\fi MR }
 \providecommand{\MRhref}[2]{%
   \href{http://www.ams.org/mathscinet-getitem?mr=#1}{#2}
 }

\end{document}